\documentclass[11pt, a4paper]{amsart}
\usepackage{amssymb, array, amsmath, amscd, pdfpages, enumerate, amsthm, setspace,mathtools}
\usepackage[papersize={21cm,29.7cm},total={12.7cm,22.5cm},top=3.6cm,left=4.15cm]{geometry}

\usepackage[utf8]{inputenc}
\usepackage{amssymb}
\usepackage{bm}
\usepackage{overpic}
\usepackage[colorlinks=true,allcolors=magenta]{hyperref}
\usepackage{xcolor}
\usepackage{graphicx} 
\usepackage{ifthen}
\usepackage[capitalize,noabbrev]{cleveref}

\newcommand{\tp}{}

\newcommand{\BNsp}[4]{(#1,#2,#3,#4)}
\newcommand{\Dom}{D}
\newcommand{\eps}{\varepsilon}

\DeclareMathOperator{\re}{Re}

\DeclareMathOperator{\diag}{diag}

\newcommand{\mcA}{{\mathfrak{A}}}     
\newcommand{\mcB}{{\mathfrak{B}}}     
\newcommand{\mcC}{{\mathfrak{C}}}     
\newcommand{\mcK}{{\mathfrak{K}}}     

\newcommand*{\C}{{\mathbb{C}}}     
\newcommand*{\R}{{\mathbb{R}}}     
\newcommand*{\Z}{{\mathbb{Z}}}     
\newcommand*{\N}{{\mathbb{N}}}     
\newcommand*{\Hloc}[1]{H^{#1}_{loc}}
\newcommand{\conj}[1]{\overline{#1}}

\newcommand*{\Lin}{{\mathcal{L}}}   
\newcommand{\ran}{{\mathcal{R}}}   
\renewcommand{\ker}{{\mathcal{N}}}

\newcommand*{\abs}[1]{\lvert#1\rvert}
\newcommand*{\norm}[1]{\lVert#1\rVert}
\newcommand*{\set}[1]{\{#1\}}
\newcommand*{\setm}[2]{\{\,#1\mid#2\,\}}   
\newcommand*{\iprod}[2]{\langle#1,#2\rangle}    
\newcommand*{\Iprod}[2]{\left\langle#1,#2\right\rangle}    
\newcommand*{\Norm}[2][default]{\ifthenelse{\equal{#1}{default}}{\left\lVert#2\right\rVert}{\ldelim{#1}{\lVert}#2\rdelim{#1}{\rVert}}}

\newcommand*{\ldelim}[2]{\csname#1l\endcsname#2}   
\newcommand*{\rdelim}[2]{\csname#1r\endcsname#2}   
\newcommand*{\mdelim}[2]{\csname#1m\endcsname#2}   
 
\newcommand*{\Lp}[1][p]{L^{#1}}

\newcommand*{\Lploc}[1][p]{L^{#1}_{\text{loc}}}

  \newcommand{\pmat}[1]{\begin{bmatrix}#1\end{bmatrix}}
\newcommand{\pmatsmall}[1]{\begin{bsmallmatrix}#1\end{bsmallmatrix}}

\newcommand{\eq}[1]{\begin{align*}#1\end{align*}}
\newcommand{\eqn}[1]{\begin{align}#1\end{align}}
\newcommand{\ieq}[1]{$#1$}

\newcommand{\gs}{\sigma}
\newcommand{\ga}{\alpha}
\newcommand{\gb}{\beta}
\renewcommand{\gg}{\gamma}
\newcommand{\gd}{\delta}
\newcommand{\gl}{\lambda}
\newcommand{\gw}{\omega}

\newcommand{\inv}{^{-1}}

\newcommand{\polpar}{\ga}

\newcommand{\citel}[2]{\cite[#2]{#1}}

\renewcommand{\pmat}[1]{\begin{bmatrix}#1\end{bmatrix}}
\renewcommand{\pmatsmall}[1]{\begin{bsmallmatrix}#1\end{bsmallmatrix}}

\newcommand{\F}{\mathbb{F}}
\newcommand{\Pt}[1][t]{\mathbf{P}_{#1}}

\renewcommand{\Dom}{D}

\renewcommand{\AA}{\mathfrak{A}}
\newcommand{\BB}{\mathfrak{B}}
\newcommand{\CC}{\mathfrak{C}}
\newcommand{\KK}{\mathfrak{K}}

\renewcommand{\pmat}[1]{\begin{bmatrix}#1\end{bmatrix}}
\renewcommand{\pmatsmall}[1]{\begin{bsmallmatrix}#1\end{bsmallmatrix}}

\renewcommand{\ran}{\textup{Ran}}
\renewcommand{\ker}{\textup{Ker}}

\newcommand*{\dda}[3][1]{\ifthenelse{\equal{#1}{1}}{\frac{d#3}{d#2}}{\frac{d^{#1}#3}{d#2^{#1}}}}
\newcommand*{\ddb}[2][1]{\ifthenelse{\equal{#1}{1}}{\frac{d}{d#2}}{\frac{d^{#1}}{d#2^{#1}}}}
\newcommand*{\pd}[3][1]{\ifthenelse{\equal{#1}{1}}{\frac{\partial{#2}}{\partial{#3}}}{\frac{\partial^{#1}{#2}}{\partial#3^{#1}}}}
\newcommand*{\pdb}[2][1]{\ifthenelse{\equal{#1}{1}}{\frac{\partial}{\partial{#2}}}{\frac{\partial^{#1}}{\partial#2^{#1}}}}

\newcommand{\zinf}[1][0]{[#1,\infty)}
\newcommand{\zt}[1][t]{[0,#1)}
\newcommand{\zabr}[2]{(#1,#2]}

\newtheorem{theorem}{Theorem}[section]
\newtheorem{lemma}[theorem]{Lemma}
\newtheorem{proposition}[theorem]{Proposition}
\newtheorem{corollary}[theorem]{Corollary}
\theoremstyle{definition}
\newtheorem{definition}[theorem]{Definition}

\newtheorem{remark}[theorem]{Remark}

\numberwithin{equation}{section}

\begin{document}

\title{Dynamic Stabilisation of Boundary Control Systems}

\author[M.~Fkirine]{Mohamed Fkirine}
\address[M.~Fkirine]{Mathematics Research Centre, Tampere University, P.O.~ Box 553, 33101 Tampere, Finland}
 \email{mohamed.fkirine@tuni.fi}

\author[L.~Paunonen]{Lassi Paunonen}
\address[L.~Paunonen]{Mathematics Research Centre, Tampere University, P.O.~ Box 553, 33101 Tampere, Finland}
 \email{lassi.paunonen@tuni.fi}

\thanks{This work was supported by the Research Council of Finland Grant number 349002.}

\thispagestyle{plain}

\begin{abstract}
We design observer-based controllers to stabilise abstract linear boundary control systems on Hilbert spaces.
Our main results introduce conditions for exponential, strong, and polynomial stability, and 
 establish external well-posedness of the closed-loop system.
We design controllers for a one-dimensional wave equation, a two-dimensional wave equation with distributed control and observation, and a non-uniform SCOLE model.
\end{abstract}



\maketitle

\section{Introduction}

In this paper we design controllers for controlled linear partial differential equations (PDEs) and infinite-dimensional linear systems. Our main interest is in stabilising PDE systems with boundary inputs. Majority of such models can be represented as \emph{abstract boundary control systems}~\cite{Sal87a,MalSta06} of the form%
\begin{subequations}%
\label{eq:plantintro}%
\label{eq:BCS}
\eqn{
\dot x(t) &= \mcA x(t) + B_i u(t) \label{eq:plantintroState}, \qquad x(0) = x_0\\
\mcB x(t) &= Qu(t) \label{eq:plantintroInput}\\
y(t) &= \mcC x(t)
}
\end{subequations}
on a Hilbert space $X$. Here~\eqref{eq:plantintroInput} describes the effect of the boundary input on the system, and the second term on the right-hand side of~\eqref{eq:plantintroState} models the possible in-domain control inputs.

We stabilise the boundary control system~\eqref{eq:plantintro} using an
 observer-based dynamic  feedback controller of the form
\begin{subequations}
\label{eq:controllerintro}
\eqn{
\dot{\hat x}(t) &= \mcA \hat x(t) + B_i u(t) + L_i (\hat y(t) - y(t)), \qquad \hat x(0) = \hat x_0\\
\mcB \hat x(t) &= Qu(t) + L(\hat y(t) - y(t))  \label{eq:controllerintroInput}\\
\hat y(t) &= \mcC \hat x(t)\\
u(t) &= \mcK \hat x(t).
}
\end{subequations}
Here $L$ and $L_i$ are output injection gains which facilitate the state estimation and $\KK$ is a stabilising state feedback gain operator.
Observer-based stabilisation of infinite-dimensional systems with boundary inputs has been studied earlier in the setting of regular linear systems in~\cite{WeiCur97} and well-posed linear systems in~\cite{Sta98}. In this work we utilise the framework of boundary control systems which often have a beneficially close connection to the underlying PDE model. In particular, the boundary trace operator $\BB$ in~\eqref{eq:plantintro} is easier to construct than the input operator in the state equations of a well-posed system. In addition, our framework also facilitates stabilisation of systems which are not well-posed.

Hyperbolic partial differential equations with boundary inputs and in-domain measurements (or vice versa) can only rarely be stabilised exponentially. We show that under mild assumptions our controller~\eqref{eq:controllerintro} instead achieves \emph{polynomial} or \emph{non-uniform stability}~\cite{BatDuy08,BorTom10,RozSei19}.
Indeed, the separation principle of observer-based stabilisation implies that closed-loop stability is based on the stability of the observer error dynamics, described by a strongly continuous semigroup $T_L$, and the stability of the system under state feedback, described by a semigroup $T_K$.
In a hyperbolic PDE with boundary control and in-domain observation, we may achieve exponential stability of $T_K$, but not of $T_L$. We show that if one of the semigroups is exponentially stable and the other one is non-uniformly stable, then the closed-loop system is non-uniformly stable, and we obtain a rate of decay for its classical state trajectories.

Our boundary control system~\eqref{eq:plantintro} differs slightly from the typical form where $Q=I$ and $B_i=0$. Our class with the bounded operator  $B_i$ can be used to study PDEs with control both at the boundary and the in interior of the spatial domain. Moreover, for parabolic systems --- even with boundary measurements --- we may often use an output injection described by a bounded operator, and this way the class~\eqref{eq:plantintro} is large enough to also contain the \emph{controllers} of classical boundary control systems. The role of $Q$ is to model the situation where the control input $u(t)$ does not affect all boundary conditions of the PDE. These uncontrolled boundary conditions could alternatively be included in the domain of $\AA$, but the formulation in~\eqref{eq:plantintroInput} facilitates their use for boundary output injections in~\eqref{eq:controllerintroInput}.

Our two main stabilisation results are for systems~\eqref{eq:plantintro} with different levels of  external well-posedness defined in terms of boundedness of the input, output, and input-output maps. Under stronger well-posedness assumptions our results also establish external well-posedness of the closed-loop system.
An important feature of observer-based stabilisation is that the Luenberger-type controller may not have well-defined dynamics when disconnected from the system. Instead, the controller should be considered to have \emph{an internal loop}~\cite{WeiCur97,CurWei96} which can be broken into an additional input-output pair to guarantee well-defined dynamics of the controller. We show that our controller~\eqref{eq:controllerintro}  has an internal loop in the sense of~\cite{WeiCur97,CurWei96}.

We illustrate the controller design for PDE systems by considering three concrete models. We first consider a two-dimensional wave equation with in-domain control and observation on different parts of the rectangular spatial domain. In the case where only one of these subdomains satisfies the Geometric Control Condition, the system cannot be stabilised exponentially, but we achieve polynomial stability. We also stabilise a one-dimensional wave equation with non-collocated boundary input and output and establish external well-posedness of the resulting closed-loop system. Finally, we stabilise a non-uniform SCOLE model~\cite{LitMar88b} consisting of an undamped beam equation with tip mass dynamics. Since the control acts through the finite-dimensional part of the model, exponential stabilisation is not possible, but our controller stabilises the system polynomially.

Observer design and observer-based stabilisation of infinite-dimensional linear systems are understood well in the case of systems with bounded input and output operators~\citel{CurZwa20book}{Sec.~8.3}. These techniques have also been extended to systems covering PDEs with boundary control and observation, most notably for Pritchard--Salamon systems in~\cite{Cur90,Log93}, regular linear systems in~\cite{WeiCur97}, and for well-posed systems in~\cite{Sta98}. In addition, finite-dimensional controllers have been introduced for parabolic systems with unbounded input and output operators in~\citel{LasTri00book}{Sec.~4B} and second order observers have been designed for abstract wave equations in~\cite{Dem04}.
Also abstract boundary control systems were treated in~\cite{Aks20} by converting them first to their state space representation.
For abstract boundary control systems the observer design problem without stabilisation has been studied in~\cite{VriKee07,VriKee10} (with focus on bilinear systems) and in~\cite{Emi20,Emi21}.  
We extend and unify these earlier studies of observer-based stabilisation of boundary control systems.
Moreover, observer-based polynomial stabilisation has previously only been studied for a one-dimensional wave equation in~\cite{PauHumSIAMCT19}, and we extend the approach in~\cite{PauHumSIAMCT19} to the general class of boundary control systems.
Besides the  abstract infinite-dimensional setting, observer-based stabilisation has been investigated actively for PDE models with boundary control and observation, e.g., in~\cite{SmyKrs05} using PDE backstepping, in~\cite{GuoWan08b} using spectral analysis, and in~\cite{MatNid22} using linear matrix inequalities. In our view, the class of abstract boundary control systems offers a convenient framework for controller design, because the representation in the form~\eqref{eq:plantintro} is relatively easy to construct compared to the regular or well-posed linear system settings, and because the controller~\eqref{eq:controllerintro} has an immediate interpretation as a PDE system. The abstract framework and our main results allow taking advantage of the separation principle to reduce the controller design problem to the design of the stabilising feedback and output injection gains for the PDE model.

The paper is organised as follows. In Section~\ref{sec:BCS} we formulate background results on the well-posedness, stability and feedback for boundary control systems. Our main results on controller design for abstract systems are presented in Section~\ref{sec:AbstractController}. In Section~\ref{sec:PDEs} we apply our results in the stabilisation of PDE systems.

\subsection*{Notation} 
If $X$ and $Y$ are Banach spaces and $A:\Dom(A)\subset X\rightarrow Y$ is a linear operator we denote by $\Dom(A)$, $\ker(A)$, and $\ran(A)$ the domain, the kernel, the and range of $A$, respectively. The space of bounded linear operators from $X$ to $Y$ is denoted by $\Lin(X,Y)$ and we write $\Lin(X)$ for $\Lin(X,X)$. If $A: \Dom(A)\subset X\rightarrow X$, then $\gs(A)$
and $\rho(A)$ denote the spectrum
and the \mbox{resolvent} set of $A$, respectively. 
The inner product on a Hilbert space is denoted by $\iprod{\cdot}{\cdot}$.
For $T\in \Lin(X)$ on a Hilbert space $X$ we define $\re T = \frac{1}{2}(T+T^\ast)$, and for $c\ge0$ we write $T\ge cI$ when $T-cI\ge 0$.
The growth bound of a strongly continuous semigroup $T$ is denoted by $\gw_0(T)$. 
We denote $f(x)\lesssim g(x)$ if there exists $M>0$ such that $f(x)\le Mg(x)$ for all parameters $x$.
The truncation of a function $f:I\to X$ defined on an interval $I\supset [0,t]$ to the interval $[0,t]$ is denoted by $\Pt f:[0,t]\to X$.
For $\gb\in\R$ we define $\C_{\gb}^+ = \setm{\gl\in\C}{\re \gl>\gb}$ and 
$H_\infty(\C_{\gb}^+;X)=\setm{f:\C_\gb^+\to X}{f \mbox{ analytic and } \sup_{\re\gl>\gb}\norm{f(\gl)}<\infty}$.

\section{Boundary Nodes}
\label{sec:BCS}

In this section we present background results for abstract boundary control systems of the form~\eqref{eq:BCS} with a state space $X$, an input space $U$, and an output space $Y$.
Our main assumption is that the operators of the system form a \emph{boundary node} defined below.

\begin{definition}\label{def:Bnode}
	The tuple $(\mcB,\mcA,\mcC,Q,B_i)$ is said to be a \emph{boundary node} on the Hilbert spaces $\BNsp{U_b}{U}{X}{Y}$ if the following hold.
	\begin{enumerate}
		\item[(i)] $\AA: \Dom(\AA)\subset X\to X$, $\mcB\in \mathcal{L}(D(\mcA),U_b)$, $\mcC\in \mathcal{L}(D(\mcA),Y)$, $Q\in \Lin(U,U_b)$, and $B_i\in \Lin(U,X)$.
		\item[(ii)] The restriction $A=\mcA\vert_{\ker(\mcB)}$ with domain $\Dom(A)=\ker(\BB)$ generates a strongly continuous semigroup on $X$.
		\item[(iii)]   $\mcB$ has a bounded right-inverse, i.e., there exists 
     $\mcB^r\in\mathcal{L}(U_b,D(\mcA))$
    such that $\mcB\mcB^r=I$.
	\end{enumerate}	
\end{definition}

We define the \emph{transfer functions} $H: \rho(A)\to \mathcal L(U,X)$ and $P: \rho(A)\to \mathcal L(U,Y)$ of the boundary node so that for $\gl\in \rho(A)$ and $u\in U $ we have $H(\gl)u := x$ and $P(\gl)u := \mcC x$ when $x\in \Dom(\mcA) $ is the unique element such that $(\gl - \mcA)x=B_i u$ and $\mcB x= Qu$. The existence and analyticity of the transfer functions follow from~\citel{NicPau25}{Prop. 2.4}.
We also note that $H(\gl)=H_b(\gl)Q + (\gl-A)\inv B_i$ and $P(\gl)=P_b(\gl)Q + \mcC(\gl-A)\inv B_i$, where $H_b$ and $P_b$ are the transfer functions of the boundary node $(\mcB,\mcA,\mcC,I,0)$ on $\BNsp{U_b}{U_b}{X}{Y}$.
We have $H(\gl)\in \Lin(U,\Dom(\mcA))$, $\mcB H(\gl)=Q$, $P(\gl)=\mcC H(\gl)$ and $(\gl-\mcA)H(\gl)=B_i$ for $\gl\in\rho(A)$ by~\citel{NicPau25}{Prop.~2.4}.

\begin{definition}
\label{def:Solutions}
Let $(\mcB,\mcA,\mcC,Q,B_i)$ be a boundary node on $\BNsp{U_b}{U}{X}{Y}$.
A triple $(x,u,y)$ is called a \emph{classical solution of} \eqref{eq:BCS} (\emph{on} $\zinf$) if
\begin{itemize}
    \item $x\in C^1(\zinf;X)$ satisfies $x(0)=x_0$, $u\in C(\zinf;U)$, and $y\in C(\zinf;Y)$
    \item $x(t)\in \Dom(\mcA)$ and 
     \eqref{eq:BCS} hold for all $t\geq 0$.
\end{itemize}
A triple $(x,u,y)$ is called a \emph{generalised solution of}~\eqref{eq:BCS} (\emph{on} $\zinf$) if 
\begin{itemize}
    \item $x\in C(\zinf;X)$ satisfies $x(0)=x_0$, 
    $u\in\Lploc[2](0,\infty;U)$, and $y\in\Lploc[2](0,\infty;Y)$
    \item there exist classical solutions $(x_k,u_k,y_k)$, $k\in\N$,  of  \eqref{eq:BCS} on $\zinf$ such that
$(\Pt[\tau] x_k,\Pt[\tau] u_k,\Pt[\tau] y_k)\tp\to (\Pt[\tau] x,\Pt[\tau] u,\Pt[\tau] y)\tp$  for every $\tau>0$ 
as $k\to \infty$ in $C([0,\tau];X)\times L^2(0,\tau;U) \times L^2(0,\tau;Y)$.
\end{itemize}
\end{definition}

In Section~\ref{sec:AbstractController} we are interested in designing a controller which stabilises the original unstable system~\eqref{eq:BCS} in such a way that the closed-loop semigroup is either exponentially stable, strongly stable, or \emph{polynomially stable}~\cite{BatDuy08,BorTom10}.
A polynomially stable semigroup is also strongly stable by~\citel{BatDuy08}{Thm.~1.1}.

\begin{definition}
\label{def:PolStab}
A strongly continuous semigroup $T$ generated by $A: \Dom(A)\subset X\to X$ on a Hilbert space $X$ is \emph{polynomially stable (with exponent $\ga>0$)} if
$T$ is bounded, i.e., $\sup_{t\ge0}\norm{T(t)}<\infty$,
 and if
 there exist $M,t_0>0$ such that 
\eq{
\norm{T(t)x} \le \frac{M}{t^\polpar}\left( \norm{Ax}+\norm{x} \right) , \qquad  t\ge t_0, \ x\in \Dom(A).
}
\end{definition}

\subsection{Well-Posedness of Boundary Nodes}

When $(\mcB,\mcA,\mcC,Q,B_i)$ is a boundary node on $\BNsp{U_b}{U}{X}{Y}$, we have from~\citel{MalSta06}{Lem. 2.6} that for every
 $x_0\in \Dom(\mcA)$ and $u\in C^2([0,\tau];U)$ satisfying $\mcB x_0=Qu(0)$ 
the equation~\eqref{eq:BCS} has a unique classical solution $(x,u,y)$.
We use these solutions to define the \emph{input map} $\Phi_\tau : C_\ell^2([0,\tau];U)\to X$, \emph{output map} $\Psi_\tau: \Dom(\mcA)\to L^2(0,\tau;Y)$, and \emph{input-output map} $\F_\tau: C_\ell^2([0,\tau];U)\to L^2(0,\tau;Y)$, where $\tau>0$ and $C_\ell^2([0,\tau];U):=\setm{f\in C^2([0,\tau];U)}{u(0)=0}$.
More precisely, we define $\Phi_\tau u =  x(\tau)$ and $\F_\tau u=\Pt[\tau] y$, where $(x,u,y)$ is the classical solution of~\eqref{eq:BCS} corresponding to $u\in C_\ell^2([0,\tau];U)$ and $x_0=0$. Moreover, we define $\Psi_\tau x_0= \Pt[\tau]y$ where $(x,u,y)$ is the classical solution of~\eqref{eq:BCS} with $u\equiv 0$ and $x_0\in \Dom(A)$.
We define the well-posedness of a boundary node following~\citel{JacZwa12book}{Def.~13.1.3} and~\cite{Sta05book}.

\begin{definition}
\label{def:BCSwellposed}
Let $(\mcB ,\mcA ,\mcC ,Q,B_i)$ be a boundary node on $\BNsp{U_b}{U}{X}{Y}$.
\begin{itemize}
\item The boundary node is \emph{(externally) well-posed} if there exist $\tau>0$ and $M_\tau>0$ such that
for all $x_0\in \Dom(\mcA)$ and $u\in C^2([0,\tau];U)$ with $\mcB x_0=Qu(0)$ the classical solution $(x,u,y)$ of~\eqref{eq:BCS} satisfies
\eq{
\norm{x(\tau)}_X^2 + \int_0^\tau \norm{y(t)}_Y^2dt
\leq M_\tau \left( \norm{x_0}_X^2 + \int_0^\tau \norm{u(t)}_U^2 dt \right).
}
\item
The boundary node has a \emph{well-posed input map} if there exist $\tau>0$ and $M_\tau>0$ such that
$\norm{\Phi_\tau u}\le M_\tau \norm{u}_{L^2(0,\tau)}$ for all $u\in C_\ell^2([0,\tau];U)$.
\item
The boundary node has a \emph{well-posed output map} if there exist $\tau>0$ and $M_\tau>0$ such that $\norm{\Psi_\tau x_0}_{L^2(0,\tau)} \le M_\tau \norm{x_0}_X$ for all $x_0\in \Dom(A)$.
\end{itemize}
\end{definition}

By definition, a well-posed boundary node has a well-posed input map and a well-posed output map.
The results in~\cite{MalSta06} and the theory of well-posed linear systems~\cite{Sta05book,TucWei14} show that if one of the estimates in \cref{def:BCSwellposed} holds for some $\tau>0$, then it also holds for every $\tau>0$ (with a modified constant $M_\tau>0$).
If a boundary node has a well-posed input map, then $\Phi_\tau$ extend to operators $\Phi_\tau\in \Lin(L^2(0,\tau;U),X)$, and if it has a well-posed output map, then $\Psi_\tau$ extend to operators $\Psi_\tau\in \Lin(X,L^2(0,\tau;Y))$.
Finally, if the boundary node is well-posed, then $\F_\tau$ extend to operators $\F_\tau\in \Lin(L^2(0,\tau;U),L^2(0,\tau;Y))$. 
In the latter case the tuple $(T,\Phi,\Psi,\F)$ with $\Phi=(\Phi_\tau)_{\tau\ge 0}$, $\Psi=(\Psi_\tau)_{\tau\ge 0}$, and $\F=(\F_\tau)_{\tau\ge 0}$ is a well-posed linear system in the sense of~\cite{TucWei14}.

\begin{proposition}
\label{prp:WPBCSsolutions}
Let $(\mcB ,\mcA ,\mcC ,Q,B_i)$ be a well-posed boundary node on $\BNsp{U_b}{U}{X}{Y}$ and let $(T,\Phi,\Psi,\F)$ be the associated well-posed linear system.
If $x_0\in \Dom(\mcA)$ and $u\in \Hloc{1}(0,\infty;U)$ satisfy $\mcB x_0= Qu(0)$, 
then~\eqref{eq:BCS} has a unique classical solution $(x,u,y)$ on $\zinf$. This solution is defined by
\begin{subequations}
\label{eq:BCSmildsol}
\eqn{
x(t) &= T(t)x_0 + \Phi_t \Pt u,  &&t\ge 0\\
\Pt y &= \Psi_t x_0 + \F_t \Pt u,  &&t\ge0
}
and we have $y\in \Hloc{1}(0,\infty;Y)$.
If $x_0\in X$ and $u\in \Lploc[2](0,\infty;U)$,  
then~\eqref{eq:BCS} has a unique generalised solution $(x,u,y)$ on $\zinf$ which is defined by~\eqref{eq:BCSmildsol}.
\end{subequations}
\end{proposition}

\begin{proof}
By~\citel{MalSta06}{Thm.~2.3} the boundary node $(\mcB,\mcA,\mcC,Q, B_i)$ defines a \emph{system node}~\citel{Sta05book}{Sec.~4.7} 
\eqn{
\label{eq:sysnode}
S= \pmat{\mcA\\ \mcC} \pmat{I\\ \mcB}\inv \pmat{I&0\\0&Q} + \pmat{0&B_i\\0&0}
}
on the spaces $(U,X,Y)$
with domain $\Dom(S) =  \setm{(x,u)\tp \in \Dom(\mcA)\times U}{\mcB x = Qu}$.
The result~\citel{MalSta06}{Thm.~2.3} and
our assumption that the boundary node $(\mcB,\mcA,\mcC,Q, B_i)$ is well-posed implies that $S$ is well-posed in the sense of~\citel{Sta05book}{Def.~4.7.2} and its associated well-posed system is exactly $\Sigma=(T,\Phi,\Psi,\F)$.
The claims follow from~\citel{TucWei14}{Prop.~4.16 \& 4.17} and the relationship between the boundary node and the system node $S$ in~\citel{MalSta06}{Thm.~2.3}.
\end{proof}

\begin{remark}
\label{rem:WPbnodetransfun}
If $(\mcB,\mcA,\mcC,Q,B_i)$ is a well-posed boundary node, then~\citel{MalSta06}{Thm.~2.3} and the results in~\citel{TucWei14}{Sec.~3 \& 4} imply that for any $\gw>\gw_0(T)$ its transfer functions $H$ and $P$ satisfy $P\in H_\infty(\C_{\gw}^+; \Lin(U,Y))$ and
$\norm{H(\gl)} \le M_1(\re \gl - \gw)^{-1/2}$ and $\norm{\CC(\gl-A)\inv} \le M_2(\re \gl - \gw)^{-1/2}$
for some constants $M_1,M_2>0$ and for all $\gl\in \C_\gw^+$.
\end{remark}

\begin{remark}
\label{rem:DistributedInputImap}
If a system does not have boundary inputs, then in the boundary node $(\mcB,\mcA,\mcC,Q,B_i)$ we can let $U_b=\set{0}$, in which case $\mcB=0$ and $Q=0$. Then $A=\mcA\vert_{\ker(\mcB)}=\mcA$ and the system~\eqref{eq:BCS} has the form
\eq{
\dot x(t) &= Ax(t) + B_i u(t), \qquad x(0)=x_0\\
y(t) &= \mcC x(t).
}
In this case the boundary node automatically has a well-posed input map. 
Analogously, if $\mcC\in \Lin(X,Y)$, then $(\mcB,\mcA,\mcC,Q,B_i)$ has a well-posed output map. More generally, the well-posedness of the output map means that   $C:=\mcC\vert_{\ker(\mcB)}$ is an admissible observation operator for the semigroup generated by $A$~\citel{TucWei09book}{Sec.~4.3}.
\end{remark}

The next result can be used in proving well-posedness of a boundary node.

\begin{lemma}
\label{lem:IPBCSWP}
Let 
 $(\mcB,\mcA,\mcC,I,0)$ be a boundary node on the Hilbert spaces $\BNsp{U}{U}{X}{U}$ with transfer function $P$ and
 let $K\in \Lin(U)$ be such that $\re K\ge cI$ for some $c>0$.
If 
\eqn{
\label{eq:BCSIP}
\re \iprod{\mcA x}{x}\le \re \iprod{\mcB x}{\mcC x}, \qquad x\in \Dom(\mcA),
}
then $A$ generates a contraction semigroup on $X$ and the following hold.
\begin{itemize}
\setlength{\itemsep}{1ex}
 \item[\textup{(a)}] 
$(\mcB,\mcA,\mcC,I,0)$ is well-posed if and only if 
 $\sup_{s\in\R}\norm{P(\gw+is)}<\infty$ for some $\gw>0$.
 \item[\textup{(b)}] 
 $(\mcB+K\mcC,\mcA,\mcC,I,0)$ is a well-posed boundary node.
 \item[\textup{(c)}]  $(I+KP(\cdot))\inv\in H_\infty(\C_0^+; \Lin(U))$.
\end{itemize}
\end{lemma}

\begin{proof}
Condition~\eqref{eq:BCSIP} implies that $A$ is dissipative, and thus the semigroup generated $A$ is contractive.
In (a), the ``only if'' part follows from \cref{rem:WPbnodetransfun}.
Assume now that $P\in H_\infty(\C_{\gw}^+; \Lin(U))$ for some $\gw>0$.
By~\citel{MalSta06}{Thm.~2.3} the boundary node $(\mcB,\mcA,\mcC,I, 0)$ defines a system node~\citel{Sta05book}{Sec.~4.7} 
\eq{
S= \pmat{\mcA\\ \mcC} \pmat{I\\ \mcB}\inv 
}
on the spaces $(U,X,U)$
with domain $\Dom(S)= \ran(\pmatsmall{I\\ \mcB}) =  \setm{(x,u)\tp \in \Dom(\mcA)\times U}{\mcB x = u}$.
The semigroup generator of $S$ is exactly $A$
and the transfer function of $S$ coincides with $P$  on $\C_{0}^+$.
Moreover, the property~\eqref{eq:BCSIP} implies that for every $(x,u)\tp\in \Dom(S)$ we have $\mcB x = u$ and $\pmatsmall{I\\\mcB}\inv \pmatsmall{x\\ u}=x$ and therefore~\eqref{eq:BCSIP} implies
\eq{
\re\Iprod{\pmat{x\\u}}{\pmat{I&0\\0&-I} S \pmat{x\\u}}_{X\times U}
= \re \iprod{x}{\mcA x}_X - \re \iprod{u}{\mcC x}_U \le 0.
}
Thus~\citel{Sta02}{Thm.~4.2} implies that $S$
 is impedance passive in the sense of~\citel{Sta02}{Def.~4.1}.
Our boundedness assumption on $P$ and~\citel{Sta02}{Thm.~5.1} imply that $S$ is a well-posed system node in the sense of~\citel{Sta02}{Def.~2.6}.
This together 
with the structure of the system node $S$ implies that $(\mcB,\mcA,\mcC,I,0)$ is well-posed.

The result in part (b) follows from~\citel{NicPau25}{Prop.~2.6} and part (a).

For part (c), let $\gl\in\C_0^+$ be arbitrary. We have $\re P(\gl)\ge 0$ by~\citel{NicPau25}{Prop.~2.4(b)}. Since $\re K\ge cI$ with $c>0$, $K$ is boundedly invertible and $\re K\inv \ge c\norm{K}^{-2}I$ by~\citel{Pau19}{Lem.~A.1(a)}. Thus  $\re (K\inv + P(\gl))\ge c\norm{K}^{-2} I$ and~\citel{Pau19}{Lem.~A.1(a)} implies that $K\inv+P(\gl)$ is boundedly invertible and $\norm{(K\inv+P(\gl))\inv}\le \norm{K}^2/c$. Since $\gl\in \C_0^+$ was arbitrary, the identity $I+KP(\gl)=K(K\inv + P(\gl))$ implies the claim.
\end{proof}

\subsection{Feedback Theory for Boundary Nodes}
In this section we investigate well-posedness and stability of boundary nodes arising from feedback. Our first result concerns the system~\eqref{eq:BCS} under output feedback $u(t)=Ky(t)$.

\begin{proposition}
\label{prp:BCSfeedback}
Assume that $(\mcB,\mcA,\mcC,Q,B_i)$ is a well-posed boundary node on the Hilbert spaces $\BNsp{U_b}{U}{X}{Y}$ with transfer function $P$. If $K\in \Lin(Y,U)$ and $(I-KP(\cdot ))\inv \in H_\infty(\C_{\gb}^+;\Lin(U))$ for some $\gb\in \R$, then $(\mcB-QK\mcC,\mcA+B_iK\mcC,\mcC,Q,B_i)$ is a well-posed boundary node $(U_b,X,Y,U)$.
\end{proposition}

\begin{proof}
The proof is based on the feedback theory of well-posed linear systems~\cite{Wei94}.
Let $H$ and $P$ be the 
 transfer functions of  $(\mcB,\mcA,\mcC,Q, B_i)$ and define $A_K=(\mcA+B_iK\mcC)\vert_{\ker(\mcB-QK\mcC)}$ with domain $\Dom(A_K)=\ker(\mcB-QK\mcC)$.
We can assume $\gb>\gw_0(T)$, where $T$ is the semigroup generated by $A$.
We will first show that $\C_{\gb}^+\subset \rho(A_K)$ and that $\mcB -QK\CC$ has a bounded right inverse.
For $\gl\in\C_{\gb}^+$, $I-KP(\gl)$ (and thus also $I-P(\gl)K$) is boundedly invertible by assumption and we can define
\eq{
R(\gl)= (\gl-A)\inv + H(\gl)(I-KP(\gl))\inv K\mcC (\gl-A)\inv.
}
A straightforward computation using the properties of boundary nodes and the transfer function $H$ in~\citel{NicPau25}{Prop.~2.4} can be used to verify that $\ran(R(\gl))\subset \Dom(A_K)$, and that $(\gl-A_K)R(\gl)=I$ and $R(\gl)(\gl-A_K)x =x$ for all $x\in \Dom(A_K)$. Thus $\gl\in\rho(A_K)$ and $(\gl-A_K)\inv =R(\gl)$.
Since $\gl\in\C_\gb^+$ was arbitrary, we have $\C_\gb^+\subset \rho(A_K)$.
Moreover, if $\gl\in\C_\gb^+$ and $\mcB^r\in \Lin(U_b,\Dom(\mcA))$ is a right inverse of $\mcB$, then a direct computation using $\mcB H(\gl)=Q$ shows that
\eq{
\MoveEqLeft[.2] (\mcB -QK\CC)(\mcB^r + H(\gl)(I-KP(\gl))\inv K\CC \mcB^r)\\
&= I-QK\CC \mcB^r +Q(I-KP(\gl))\inv K\CC \mcB^r -QKP(\gl)(I-KP(\gl))\inv K\CC \mcB^r\\
&= I.
}
Thus $\mcB^r + H(\gl)(I-KP(\gl))\inv K\CC \mcB^r\in \Lin(U_b,\Dom(\mcA))$ is a right inverse of $\mcB-QK\CC$, as required.

Similarly as in the proof of \cref{prp:WPBCSsolutions}, $(\mcB,\mcA,\mcC,Q, B_i)$ defines the system node $S$ in~\eqref{eq:sysnode} 
on the spaces $(U,X,Y)$
with domain $\Dom(S)=   \setm{(x,u)\tp \in \Dom(\mcA)\times U}{\mcB x = Qu}$.
By~\citel{MalSta06}{Thm.~2.3} the semigroup generator of $S$ is exactly $A$,
its output operator  is $C=\mcC\vert_{\Dom(A)}$, and
its input operator $B$  satisfies
 $H(\gl)=(\gl-A)\inv B$ for $\gl\in \rho(A)$ (here $A$ is considered as its extension to $X$,
 see~\citel{TucWei09book}{Sec.~4.2} for details).
The transfer function of $S$ coincides with  $P$ on the largest right half-plane of $\C$ contained in $\rho(A)$.
The result~\citel{MalSta06}{Thm.~2.3} and
our assumption that the boundary node $(\mcB,\mcA,\mcC,Q, B_i)$ is well-posed implies that $S$ is well-posed in the sense of~\citel{Sta05book}{Def.~4.7.2} and its associated well-posed system  $\Sigma=(T,\Phi,\Psi,\F)$ is composed of the input, output, and input-output maps of the boundary node.

Since $(I-KP(\cdot ))\inv \in H_\infty(\C_{\gb}^+;\Lin(U))$,
we have from~\citel{TucWei14}{Prop.~5.15} and~\citel{Sta05book}{Thm.~7.1.8} that $K$ is an admissible feedback operator for $\Sigma$ in the sense of~\citel{Sta05book}{Def.~7.1.1}. Because of this, there exists a well-posed linear system $\Sigma^K=(T^K,\Phi^K,\Psi^K,\F^K)$ such that 
\eq{
\pmat{T_K(\tau)&\Phi_\tau^K\\ \Psi_\tau^K & \F_\tau^K} - 
\pmat{T(\tau)&\Phi_\tau\\ \Psi_\tau & \F_\tau} 
=
\pmat{T(\tau)&\Phi_\tau\\ \Psi_\tau & \F_\tau} \pmat{0&0\\0&K}
\pmat{T_K(\tau)&\Phi_\tau^K\\ \Psi_\tau^K & \F_\tau^K}  
}
for all $\tau\ge 0$.
We have from~\citel{Sta05book}{Thm.~7.4.1 \& Lem.~7.4.4(iii)} that the semigroup generator $A^K$, the input operator $B^K$, the output operator $C^K$, and the transfer function $P^K$ of the system node $S^K$ associated to $\Sigma^K$ satisfy 
\eqn{
\label{eq:SysnodeFBrel}
\MoveEqLeft[1]\pmat{(\gl-A^K)\inv & (\gl-A^K)\inv B^K \\ C^K(\gl-A^K)\inv & P^K(\gl)} \\
&= \pmat{R(\gl) & H(\gl)(I-KP(\gl))\inv
\\ 
(I-P(\gl)K)\inv \mcC (\gl-A)\inv & P(\gl)(I-KP(\gl))\inv
}
\nonumber
}
on some open right half-plane contained in $\rho(A)\cap \rho(A^K)$. The formula in particular implies that $(\gl-A^K)\inv =R(\gl)= (\gl - A_K)\inv$ for $\gl$ on some right half-plane, and thus $A_K=A^K$ is the generator of the semigroup $T_K$.
This implies that $(\mcB-QK\mcC,\mcA+B_iK\mcC,\mcC,Q,B_i)$ is a boundary node. 
Denoting the transfer functions of this boundary node by $H_K$ and $P_K$, similar computations as in the proof of~\citel{NicPau25}{Prop.~2.6} show that, for every $\gl \in \C_{\gb}^+$,
\eq{
H_K(\gl) &= H(\gl) (I-KP(\gl))\inv\\
\mcC (\gl-A_K)\inv &= (I-P(\gl)K)\inv \mcC (\gl-A)\inv \\
P_K(\gl) &= P(\gl) (I-KP(\gl))\inv
}
This and~\eqref{eq:SysnodeFBrel} imply that the system node associated to this boundary node via~\citel{MalSta06}{Thm.~2.3} is $S^K$.
In particular, the input, output, and input-output maps of the boundary node $(\mcB-QK\mcC,\mcA+B_iK\mcC,\mcC,Q,B_i)$ are exactly $\Phi_\tau^K$, $\Psi_\tau^K$, and $\F_\tau^K$, respectively. 
Since $\Sigma^K$ is a well-posed system, the boundary node $(\mcB-QK\mcC,\mcA+B_iK\mcC,\mcC,Q,B_i)$ is well-posed.
\end{proof}

The following result presents sufficient conditions for well-posedness of cascaded and coupled systems.

\begin{proposition}
\label{prp:CascCoupWP}
Let $(\mcB_1,\mcA_1,\mcC_1,Q_1, B_{i1})$ be a well-posed boundary node on $\BNsp{U_{b1}}{U_1}{X_1}{Y_1}$
 with transfer function $P_1$ and let $(\mcB_2,\mcA_2,\mcC_2,Q_2, B_{i2})$ be a well-posed boundary node on $\BNsp{U_{b2}}{U_2}{X_2}{Y_2}$ with transfer function $P_2$. 
\begin{itemize}
\item[\textup{(a)}] If $K\in \Lin(Y_2,U_1)$, then 
\eq{
\left( \pmat{\mcB_1 & -Q_1K\mcC_2 \\ 0 & \mcB_2} 
, \pmat{\mcA_1 & B_{i1} K\mcC_2 \\ 0 & \mcA_2}, \pmat{\mcC_1 & 0 \\ 0 & \mcC_2},  \pmat{Q_1 & 0 \\ 0 & Q_2}, \pmat{B_{i1} & 0 \\ 0 & B_{i2}}\right)
}
is a well-posed boundary node on $\BNsp{U_{b1}\times U_{b2}}{U_1\times U_2}{X_1\times X_2}{Y_1\times Y_2}$
\item[\textup{(b)}] If $K_1\in \Lin(Y_2,U_1)$ and $K_2\in \Lin(Y_1,U_2)$ and if $(I-K_2P_1(\cdot)K_1P_2(\cdot))\inv \in H_\infty(\C_\gb^+;\Lin(U_2))$ for some $\gb\in\R$, then 
\eq{
\left(  \hspace{-.2ex}\pmat{\mcB_1 & \hspace{-2.5ex} -Q_1K_1\mcC_2 \\ -Q_2K_2\mcC_1 &\hspace{-2.5ex} \mcB_2} \hspace{-.5ex}
,  \hspace{-.2ex} \pmat{\mcA_1 & \hspace{-2.5ex}B_{i1} K_1\mcC_2 \\ B_{i2} K_2\mcC_1 & \hspace{-2.5ex}\mcA_2} \hspace{-.5ex}, \hspace{-.2ex}\pmat{\mcC_1 &\hspace{-1.5ex}  0 \\ 0 & \hspace{-1.5ex} \mcC_2} \hspace{-.5ex}, \hspace{-.2ex} \pmat{Q_1 &\hspace{-1.5ex}  0 \\ 0 &\hspace{-1.5ex}  Q_2} \hspace{-.5ex}, \hspace{-.2ex}\pmat{B_{i1} & \hspace{-1.5ex} 0 \\ 0 & \hspace{-1.5ex} B_{i2}} \hspace{-.2ex}\right)
}
is a well-posed boundary node on $\BNsp{U_{b1}\times U_{b2}}{U_1\times U_2}{X_1\times X_2}{Y_1\times Y_2}$. 
\end{itemize}
\end{proposition}

\begin{proof}
Part (a) follows from part (b) with the choices $K_1=K$ and $K_2=0$.
To prove part (b), denote
$A_1=\mcA_1\vert_{\ker(\mcB_1)}$, 
$A_2=\mcA_2\vert_{\ker(\mcB_2)}$,
 $X=X_1\times X_2$, $U=U_1\times U_2$, $Y=Y_1\times Y_2$, and $U_b=U_{b1}\times U_{b2}$. Define $\mcA_o : \Dom(\mcA_o)\subset X\to X$, $\mcB_o\in \Lin(\Dom(\mcA_o),U_b)$, $\mcC_o\in \Lin(\Dom(\mcA_o),U)$, and $B_{io}\in \Lin(U,X)$ by $\Dom(\mcA_o)= \Dom(\mcA_1)\times \Dom(\mcA_2)$,
$\mcA_o = \diag(\mcA_1,\mcA_2)$, 
$\mcB_{io} = \diag(\mcB_1,\mcB_2)$, 
$\mcC_o = \diag(\mcC_1,\mcC_2)$, 
$Q_o = \diag(Q_1,Q_2)$ and
$B_o = \diag(B_{i1},B_{i2})$.
It is straightforward to verify that $(\mcB_o,\mcA_o,\mcC_o,Q,B_{io})$ is a well-posed boundary node on $\BNsp{U_b}{U}{X}{Y}$ and that its transfer function is $P=\diag(P_1,P_2)$ defined on $\rho(\mcA_o\vert_{\ker(\mcB_o)})=\rho(A_1)\cap \rho(A_2)$. 
Define
\ieq{
K_o=\pmatsmall{0&K_1\\K_2&0}\in \Lin(Y, U).
}
If $\gb\in\R$ is such that $(I-K_2P_1(\cdot)K_1P_2(\cdot))\inv \in H_\infty(\C_\gb^+;\Lin(U_1))$, then for every $\gl\in \C_\gb^+$ the operator
$I-K_oP(\gl)
$ 
has a bounded inverse
\eq{
(I-K_oP(\gl))\inv
= \pmat{I &\hspace{-1ex} K_1P_2(\gl)\\0 &I}
\hspace{-.5ex}\pmat{I&0\\0&\hspace{-1.5ex}(I-K_2P_1(\gl)K_1P_2(\gl))\inv} \hspace{-.5ex}\pmat{I&\hspace{-1.5ex}0\\K_2P_1(\gl) &\hspace{-1.5ex}I}\hspace{-.5ex}.
}
 \cref{rem:WPbnodetransfun} implies that $(I-K_0P(\cdot))\inv 
\in H_\infty(\C_\gb^+;\Lin(U))$ and the claim in part (b) follows from \cref{prp:BCSfeedback}.
\end{proof}

The following proposition concerning the stability properties of a cascade of two boundary nodes
$(\mcB_1,\mcA_1,\mcC_1,Q_1, B_{i1})$ and
$(\mcB_2,\mcA_2,\mcC_2,Q_2, B_{i2})$
 is the basis of our stability analysis in Section~\ref{sec:AbstractController}. 
In the result we denote by $T_1$ and $T_2$ the semigroups generated by $A_1= \mcA_1\vert_{\ker(\mcB_1)}$ and $A_2=\mcA_2\vert_{\ker(\mcB_2)}$, respectively.
Observability-type conditions for showing the polynomial stability of the semigroups $T_1$ and $T_2$ have been presented in~\cite{AnaLea14,JolLau20,ChiPau23}. 

\begin{proposition}
\label{prp:BNcascade}
Assume that $(\mcB_1,\mcA_1,\mcC_1,Q_1, B_{i1})$ is a boundary node with a well-posed input map  on $\BNsp{U_{b1}}{U_1}{X_1}{Y_1}$ and that $(\mcB_2,\mcA_2,\mcC_2,Q_2, B_{i2})$ is a boundary node with a well-posed output map on $\BNsp{U_{b2}}{U_2}{X_2}{Y_2}$. Then the operator defined by
\eq{
A = \pmat{\mcA_1 & B_{i1} \mcC_2 \\ 0 & \mcA_2}, 
\qquad \Dom(A) = \ker\left( \pmat{\mcB_1 & -Q_1\mcC_2 \\ 0 & \mcB_2} \right)
}
generates a strongly continuous semigroup $T$ on $X=X_1 \times X_2$
and we have $\gw_0(T)= \max \set{\gw_0(T_1),\gw_0(T_2)}$.
In addition, the following hold.
\begin{itemize}
\item[\textup{(a)}] If $T_1$ and $T_2$ are exponentially stable, then $T$ is exponentially stable.
\item[\textup{(b)}] If $T_1$ is strongly stable and $T_2$ is exponentially stable (or vice versa), then $T$ is strongly stable.
\item[\textup{(c)}] If $T_1$ is polynomially stable with exponent $\polpar>0$ and $T_2$ is exponentially stable (or vice versa), then $T$ is polynomially stable with exponent $\polpar>0$.
\end{itemize}
Finally, $ \rho(A) = \rho(A_1)\cap \rho(A_2) $ and
 for every $\gb>\min \set{\gw_0(T_1), \gw_0(T_2)}$ there exists $M>0$ such that 
\eqn{
\label{eq:ResNormEst}
\norm{(\gl-A)\inv} \le M (1+ \max \set{\norm{(\gl-A_1)\inv}, \norm{(\gl-A_2)\inv}}) 
}
for all $\gl\in\rho(A)$ with $\re\gl\ge \gb$.
\end{proposition}

\begin{remark}
\label{rem:NUStab}
By the choice of $\gb$, one of the resolvent norms on the right-hand side of~\eqref{eq:ResNormEst} is uniformly bounded with respect to $\gl$ with $\re\gl\ge \gb$.
This estimate implies that the claim regarding the polynomial stability of $T$ also immediately extends to \emph{non-uniform stability} in the sense of~\cite{BatDuy08,RozSei19}. Indeed, if $T_1$ is non-uniformly stable in the sense that
$T_1$ is bounded, $i\R \subset \rho(A_1)$, and $\norm{(is-A_1)\inv} \le N(s)$, $s\in \R$, for some non-decreasing function $N: \R\to \zinf[1]$ and if $T_2$ is exponentially stable, then $T$ is bounded (by part (b)), $i\R \subset \rho(A)$, and $\norm{(is-A)\inv} \lesssim N(s)$, $s\in\R$. An analogous conclusion holds if the roles of $T_1$ and $T_2$ are reversed.
\end{remark}

The proof of \cref{prp:BNcascade} is based on the following lemma. This result is a variation of~\citel{OosCur98}{Lem.~12} where the assumption of ``infinite-time admissibility'' is replaced with a stronger condition on the function $u$.

\begin{lemma}
\label{lem:InputMapStability}
Let $\Phi_t\in \Lin(L^2(0,t;U),X)$, $t\ge 0$, be well-posed input maps associated with a strongly stable semigroup $T$. If $u\in \Lploc[2](0,\infty;U)$ satisfies $(\norm{u}_{\Lp[2](k,k+1)})_{k=0}^\infty\in \ell^1(\C)$, then $\Phi_t \Pt u\to 0$ as $t\to\infty$. 
\end{lemma}

\begin{proof}
Let $u\in \Lploc[2](0,\infty;U)$ be such that $(\norm{u}_{\Lp[2](k,k+1)})_{k=0}^\infty\in \ell^1(\C)$. Let $t>1$ be arbitrary and let $t_0\in \zt[1]$ and $n\in\N$ be such that $t=t_0+n$. Moreover, let $m\in \N_0$ be such that $m< n$.
The composition property~\citel{TucWei09book}{Sec.~4.2} of the input maps $\Phi_t$ implies that 
\eqn{
\label{eq:ImapSplit}
\Phi_t \Pt u
 &= T(t-m)\Phi_m \Pt[m] u + \Phi_{t-m} \Pt[t-m]u(\cdot+m) .
}
The composition property further implies that the second term satisfies
\eq{
\MoveEqLeft  \Phi_{t-m} \Pt[t-m]u(\cdot+m) \\
&=  \Phi_{t_0}\Pt[t_0] u(\cdot+n) + \sum_{k=0}^{n-m-1} T(t+m-k-1)\Phi_1 \Pt[1]u(\cdot+k+m).
}
Since $T$ is strongly stable, there exists $M\ge 1$ such that $\norm{T(s)}\le M$, $s\ge 0$. 
Since $\norm{\Phi_{t_0}}\le \norm{\Phi_1}\le \kappa$ 
 for some $\kappa>0$
by~\citel{TucWei09book}{Sec.~4.2}, we can estimate
\eq{
\MoveEqLeft[7]  
\norm{\Phi_{t-m} \Pt[t-m]u(\cdot+m)} 
\le \kappa\norm{u(\cdot+n)}_{\Lp[2](0,t_0)} + M\kappa \hspace{-1.2ex}\sum_{k=0}^{n-m-1} \norm{u(\cdot+k+m)}_{\Lp[2](0,1)}\\
&\le  M\kappa\sum_{k=m}^n \norm{u(\cdot+k)}_{\Lp[2](0,1)}
\le  M\kappa\sum_{k=m}^\infty \norm{u}_{\Lp[2](k,k+1)}.
}
Let $\eps>0$.
Since $(\norm{u}_{\Lp[2](k,k+1)})_{k=0}^\infty\in \ell^1(\C)$ by assumption, there exists $m\in\N$ such that the right-hand side is less than $\eps/2$.
Since $\Phi_m\Pt[m]u$  is a fixed element of $X$ and $T$ is strongly stable, 
there exists $t_\eps>m$ such that 
 $\norm{T(t-m)\Phi_m \Pt[m] u}<\eps/2$ for all $t\ge t_\eps$. 
Together with~\eqref{eq:ImapSplit} these properties imply that $\norm{\Phi_t\Pt u}\to 0$ as $t\to\infty$ as claimed.
\end{proof}

\begin{proof}[Proof of Proposition~\textup{\ref{prp:BNcascade}}]
The property that $A$ generates a strongly continuous semigroup on $X$ follows from applying \cref{prp:CascCoupWP}(a) to the well-posed boundary nodes 
$(\mcB_1,\mcA_1,0,Q_1, B_{i1})$  (on $\BNsp{U_{b1}}{ U_1}{X_1}{ \set{0}}$) and $(\mcB_2,\mcA_2,\mcC_2,0, 0)$  (on $\BNsp{U_{b2}}{ \set{0}}{X_2}{U_1}$) with $K_1=I\in \Lin(U_1)$.
Let $H_1$ be the transfer function of
 $(\mcB_1,\mcA_1,0,Q_1, B_{i1})$ 
and
 denote
$A_1=\mcA_1\vert_{\ker(\mcB_1)}$, 
$A_2=\mcA_2\vert_{\ker(\mcB_2)}$,  and
 $C_2=\mcC_2\vert_{\Dom(A_2)}$.
For $\gl\in\rho(A_1)\cap \rho(A_2)$ we define
\eqn{
\label{eq:cascaderesolvent}
R(\gl) = \pmat{(\gl-A_1)\inv & H_1(\gl) C_2 (\gl-A_2)\inv \\0 & (\gl-A_2)\inv}\in \Lin(X).
}
A straightforward computation using the properties of the boundary nodes and the transfer function $H_1$ in~\citel{NicPau25}{Prop.~2.4} can be used to verify that $\ran(R(\gl))\subset \Dom(A)$, and that $(\gl-A)R(\gl)=I$ and $R(\gl)(\gl-A)x =x$ for all $x\in \Dom(A)$. Thus $\gl\in\rho(A)$ and $(\gl-A)\inv =R(\gl)$.
Conversely, if $\gl\in\rho(A)$, then for every $y=(y_1,y_2)\tp\in X$ there exists a unique $x=(x_1,x_2)\tp \in \Dom(A)$ such that $(\gl-A)x=y$. Since $A_1$ and $A_2$ are closed operators, choosing $y_1=0$ and letting $y_2\in X_2$ be arbitrary shows that $\gl\in \rho(A_2)$, and choosing $y_2=0$ and letting $y_1\in X_1$ be arbitrary shows that $\gl\in \rho(A_1)$.
Thus $\rho(A)=\rho(A_1)\cap \rho(A_2)$.

We will now estimate the resolvent of  $A$.
Let
 $\gb>\min \set{\gw_0(T_1),\gw_0(T_2)}$.
Remark~\ref{rem:WPbnodetransfun} shows that the norms $\norm{H_1(\gl)}$ and $\norm{C_2(\gl-A_2)\inv}$ are uniformly bounded with respect to $\gl\in\C$ satisfying 
$\re\gl\ge \max \set{\gw_0(T_1),\gw_0(T_2)}+1$.
The resolvent identity 
and~\citel{NicPau25}{Prop.~2.4(iii)} imply 
 $H_1(\gl)=\left[1+\gd (\gl-A_1)\inv \right]H_1(\gl+\gd)$ 
and
$C_2(\gl-A_2)\inv =C_2(\gl+\gd-A_2)\inv\left[ 1+\gd (\gl-A_2)\inv \right] $ whenever $\gl,\gl+\gd\in\rho(A_1)\cap \rho(A_2)$, and thus 
 $\norm{H_1(\gl)}\lesssim 1+\norm{(\gl-A_1)\inv}$ and $\norm{C_2(\gl-A_2)\inv}\lesssim 1+\norm{(\gl-A_2)\inv}$ for $\gl\in\rho(A_1)\cap\rho(A_2)$ with $\re \gl\ge\gb$.
This, the fact that either $\norm{(\gl-A_1)\inv}$ or $\norm{(\gl-A_2)\inv}$ is uniformly bounded with respect to $\gl\in\C$ with $\re\gl\ge \gb$, and
 the form~\eqref{eq:cascaderesolvent} of the resolvent operator $(\gl-A)\inv$ imply the estimate for $\norm{(\gl-A)\inv}$ in the claim.
 Moreover, the estimate~\eqref{eq:ResNormEst} and the Gearhart--Pr\"uss--Greiner theorem imply $\gw_0(T)\le \max \set{\gw_0(T_1),\gw_0(T_2)}$.
  Since we clearly  have
$\gw_0(T)\ge \max \set{\gw_0(T_1),\gw_0(T_2)}$, we conclude that $\gw_0(T)\ge \max \set{\gw_0(T_1),\gw_0(T_2)}$ as claimed.

To complete the proof, we will analyse the 
stability of the semigroup $T$.
The claim regarding exponential stability follows directly from $\gw_0(T)= \max \set{\gw_0(T_1),\gw_0(T_2)}$.
 A direct computation using the feedback identity for $\Sigma^K$ in the proof of \cref{prp:BCSfeedback} shows that 
for every $t\ge0$ and $x=(x_1,x_2)\tp \in X$ we have $T(t)x=(T_1(t)x_1 + \Phi_t^1\Psi_t^2x_2,T_2(t)x_2)\tp$.
If $T_1$ is strongly stable and $T_2$ is exponentially stable,
there exists $M,\gw>0$ such that $\norm{T_2(t)}\le Me^{-\gw t}$, $t\ge 0$.
 For every $x\in X_2$ the function $u\in \Lploc[2](0,\infty;U_1)$ defined so that $\Pt u= \Psi_t^2 x$ for $t\ge 0$ satisfies  $\norm{u}_{\Lp[2](k,k+1)} = \norm{\Psi_1^2 T_2(k)x}_{\Lp[2](0,1)}\lesssim e^{-\gw k}$, which implies $(\norm{u}_{\Lp[2](k,k+1)})_{k=0}^\infty \in \ell^1(\C)$. Thus  $\Phi_t^1\Psi_t^2 x= \Phi_t^1\Pt u\to 0$ as $t\to\infty$ by \cref{lem:InputMapStability}. Since $x\in X_2$ was arbitrary, this implies that $T$ is strongly stable.
We will next analyse the case where $T_1$ is exponentially stable and 
$T_2$ is strongly stable.
Let $0>\gw_1>\gw_0(T_1)$ and $M\ge 1$ be such that $\norm{T_1(t)}\le Me^{\gw_1 t}$, $t\ge 0$ and
let $x\in X_2$. The composition properties of the input and output map~\citel{TucWei09book}{Sec.~4.2 \& 4.3} imply that if  $t=t_0+n$ with $t_0\in \zt[1]$ and $n\in\N_0$, then
\eq{
\Phi_t^1 \Psi_t^2 x = \Phi_{t_0}^1 \Psi_{t_0}^2 T_2(n)x + \sum_{k=0}^{n-1}T_1(t_0+n-1-k)\Phi_1^1 \Psi_1^2 T_2(k)x.
}
If we denote $a_k=\norm{T_2(k)x}$, $k\in\N_0$, then $\norm{\Phi_{t_0}^1}\le \norm{\Phi_1^1}$ and $\norm{\Psi_{t_0}^2}\le \norm{\Psi_1^2}$ (see~\citel{TucWei09book}{Sec.~4.2 \& 4.3})
imply that
\eq{
\norm{\Phi_t^1 \Psi_t^2 x}
\lesssim a_n + \sum_{k=0}^{n-1}e^{\gw_1(n-1-k)} a_k
\lesssim  \sum_{k=0}^ne^{\gw_1(n-k)} a_k
= (a\ast b)_n,
}
where $a=(a_k)_{k=0}^\infty$ and $b= (e^{\gw_1 k})_{k=0}^\infty$.
Since $T_2$ is strongly stable,
 $a_k\to 0$ as $k\to \infty$ and thus $(a\ast b)_n\to 0$ as $n\to \infty$. This implies that 
$\norm{\Phi_t^1 \Psi_t^2 x}\to 0$ as $t\to\infty$, and since $x\in X_2$ was arbitrary,  $T$ is strongly stable.

Finally, consider the case where $T_1$  is exponentially stable and $T_2$ is polynomially stable with exponent $\polpar>0$. 
Since a polynomially stable semigroup is strongly stable,  $T$ is strongly stable and in particular uniformly bounded by part (b). Moreover, $i\R\subset \rho(A_2)$ by~\citel{BatDuy08}{Thm.~1.1}, and thus
 $i\R\subset \rho(A_1)\cap \rho(A_2)= \rho(A)$.
Exponential stability of $T_1$
 and~\citel{BorTom10}{Thm.~2.4} imply  $\norm{(is-A_1)\inv}\lesssim 1$ and $\norm{(is-A_2)\inv}\lesssim 1+\abs{s}^{1/\polpar}$ for $s\in\R$. Thus~\eqref{eq:ResNormEst} implies that $\norm{(is-A)\inv}\lesssim 1+\abs{s}^{1/\polpar}$, and~\citel{BorTom10}{Thm.~2.4} shows that $T$ is polynomially stable with exponent $\polpar$. A completely analogous argument applies in the situation where $T_1$ is polynomially stable with exponent $\polpar>0$ and $T_2$ is exponentially stable.
\end{proof}

\section{Stabilising Controller Design}
\label{sec:AbstractController}

In this section we design a stabilising controller for the system
\begin{subequations}
\label{eq:plant}
\eqn{
\dot x(t) &= \mcA x(t) + B_i u(t), \qquad x(0) = x_0\\
\mcB x(t) &= Qu(t) \\
y(t) &= \mcC x(t),
}
\end{subequations}
where $(\mcB,\mcA,\mcC,Q,B_i)$ is a boundary node on the Hilbert spaces $\BNsp{U_b}{U}{X}{Y}$.
Our observer-based controller has the form
\begin{subequations}
\label{eq:controller}
\eqn{
\MoveEqLeft[6]\dot{\hat x}(t) = (\mcA+ B_i\mcK + L_i\mcC )\hat x(t) + B_i u_1(t) + L_i (u_2(t)-y(t)) \\
(\mcB -Q\mcK - L\mcC)\hat x(t) &= Qu_1(t) + L(u_2(t)-y(t)) \\
\hat y(t) &= \mcC \hat x(t) \\
u(t) &= \mcK \hat x(t) + u_1(t) 
}
\end{subequations}
with initial state $\hat x(0) = \hat x_0$.
Here $u$ and $y$ are the input and output, respectively, of the system~\eqref{eq:plant}, 
 $u_1$ and $u_2$ are external inputs, and $\hat y$ and $u $ are the outputs of the controller (see \cref{fig:contrscheme}).

\begin{figure}[h!]
\center{
\includegraphics[width=0.65\linewidth]{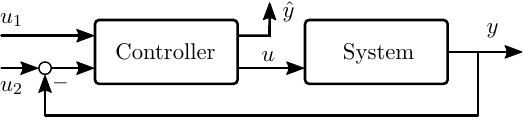}
}
\caption{The control scheme.}
\label{fig:contrscheme}
\end{figure}

The parameters of the controller are  
the state feedback gain $\mcK$  and the output injection gains $L$ and $L_i$.
Roughly stated, these parameters should be chosen so that the operators $A_K:=(\mcA+B_i\mcK)\vert_{\ker(\mcB-Q\mcK)}$ and  $A_L:=(\mcA+L_i\mcC)\vert_{\ker(\mcB-L\mcC)}$ generate stable semigroups. Our first main results in \cref{thm:ObsStab} and \cref{prp:ObsStabIP}  provide conditions for the stability and well-posedness of the closed-loop system consisting of~\eqref{eq:plant} and~\eqref{eq:controller}. Later in \cref{cor:ObsStabSGonly} we also stabilise the closed-loop system under weaker assumptions and without external well-posedness.
We note that the controller does not need to be associated to a boundary node when it is not connected to the system~\eqref{eq:plant}. 
We explain this aspect in detail after \cref{prp:ObsStabIP}.

If we define $x_e(t)=(x(t),\hat x(t))\tp\in X_e:=X\times X$,  $ u_e(t)= (u_1(t),u_2(t))\tp \in U_e:=U\times Y$, $ y_e(t)= (y(t),\hat y(t), u(t))\tp \in Y_e:=Y\times Y \times U$, and $U_{be}:=U_b\times U_b$, then the closed-loop dynamics of~\eqref{eq:plant} and~\eqref{eq:controller} are described by
\begin{subequations}
\label{eq:ObsCLsysExtIO}
\eqn{
\dot x_e(t) &= \mcA_e x_e(t) + B_{ei}  u_e(t), \qquad x_e(0)= x_{e0}\\
\mcB x_e(t) &= Q_e  u_e(t)\\
 y_e(t) &= \mcC_e x_e(t),
}
\end{subequations}
where $\mcA_e: \Dom(\mcA_e)\subset X_e\to X_e$
with domain $\Dom(\mcA_e)= \Dom(\mcA)\times \Dom(\mcA)$, 
 $\mcB_e\in \Lin( \Dom(\mcA_e), U_{be})$, $\mcC_e\in \Lin( \Dom(\mcA_e), Y_e)$, $B_{ei}\in \Lin(U_e, X)$, and $Q_e\in \Lin(U_e, U_{be})$ are 
\begin{subequations}
\label{eq:ObsCLsysExtIO_operators}
\eqn{
\mcA_e &= \pmat{\mcA & B_i \mcK \\ -L_i \mcC & \mcA + B_i\mcK + L_i \mcC}, \quad
\mcB_e = \pmat{\mcB & -Q \mcK \\ L\mcC & \mcB - Q\mcK - L \mcC},\\
\label{eq:ObsCLsysExtIO_operators2}
B_{ei} &= \pmat{B_i & 0\\ B_i & L_i}, \quad 
Q_e = \pmat{Q & 0\\ Q & L}, \quad 
\mcC_e = \pmat{\mcC & 0\\ 0& \mcC \\ 0& \mcK}.
}
\end{subequations}
Our main result below introduces conditions under which the closed-loop system is associated with a well-posed system node 
$(\mcB_e,\mcA_e, \mcC_e, Q_e,B_{ei})$ on 
$\BNsp{U_{be}}{U_e}{X_e}{Y_e}$ and 
provides conditions for the stability of the semigroup $T_e$ generated by $A_e:= \mcA_e\vert_{\ker(\mcB_e)}$.
Here we denote $A_K:=(\mcA+B_i\mcK)\vert_{\ker(\mcB-Q\mcK)}$ and $A_L:=(\mcA+L_i\mcC)\vert_{\ker(\mcB-L\mcC)}$.
Observability-type conditions for polynomial stability of the semigroups $T_K$ and $T_L$ have been presented in~\cite{AnaLea14,JolLau20,ChiPau23}. 
Moreover, \cref{lem:IPBCSWP}(c) can be used in verifying the conditions on the transfer functions $P_K$ and $P_L$ arising from collocated feedbacks and output injections.

\begin{theorem}
\label{thm:ObsStab}
Let $(\mcB ,\mcA ,  \mcC, Q, B_i)$ be a  boundary node on the Hilbert spaces $\BNsp{U_b}{U}{X}{Y}$ and let $\mcK\in \Lin(\Dom(\mcA),U)$, $L\in \Lin(Y,U_b)$, and $L_i \in \Lin(Y,X)$.
Assume that $(\mcB ,\mcA,  \pmatsmall{\mcC\\ \mcK}, [Q,L], [B_i,L_i])$ is a well-posed boundary node on $\BNsp{U_b}{U_e}{X}{Y\times U}$
and that there exists $\gb\in \R$ such that
\eq{
(I-P_K(\cdot))\inv \in H_\infty(\C_{\gb}^+;\Lin(U)) \quad \mbox{and} \quad (I-P_L(\cdot))\inv \in H_\infty(\C_{\gb}^+;\Lin(Y)),
}
where 
 $P_K$ and $P_L$ are the transfer functions of $(\mcB ,\mcA,   \mcK, Q, B_i)$ and $(\mcB ,\mcA,  \mcC, $ $ L, L_i)$, respectively.
Then $(\mcB_e,\mcA_e,\mcC_e,Q_e,B_{ei})$ is a well-posed boundary node on $\BNsp{U_{be}}{U_e}{X_e}{Y_e}$ and $A_K$ and $A_L$ generate strongly continuous semigroups $T_K$ and $T_L$, respectively. 
Moreover, the following hold.
\begin{itemize}
\item[\textup{(a)}] If $T_K$ and $T_L$ are exponentially stable, then $T_e$ is exponentially stable with growth bound $\gw_0(T_e)=\max \set{\gw_0(T_K), \gw_0(T_L)}$.
\item[\textup{(b)}] If $T_K$ is strongly stable and $T_L$ is exponentially stable (or vice versa), then $T_e$ is strongly stable.
\item[\textup{(c)}] If $T_K$ is polynomially stable with exponent $\polpar>0$ and $T_L$ is exponentially stable (or vice versa), then $T_e$ is polynomially stable with exponent $\polpar>0$.
\end{itemize}
\end{theorem}

\begin{proof}
We will establish the boundary node property and well-posedness of $(\mcB_e,\mcA_e, \mcC_e, Q_e,B_{ei})$  by applying output feedback of the form in \cref{prp:BCSfeedback} to an open-loop boundary node $(\mcB_o,\mcA_o, \mcC_e, Q_e,B_{ei})$ on the spaces $\BNsp{U_{be}}{U_e}{X_e}{Y_e}$. Indeed, if we define 
$\mcA_o: \Dom(\mcA_o)\subset X_e\to X_e$,
with domain $\Dom(\mcA_o)= \Dom(\mcA)\times \Dom(\mcA)$, 
 $\mcB_o \in\Lin( \Dom(\mcA_o), U_{be})$, and $K\in \Lin(Y_e,U_e)$ by
\eq{
\mcA_o &= \pmat{\mcA & 0 \\ 0 & \mcA }, \quad
\mcB_o = \pmat{\mcB & 0 \\ 0 & \mcB }, \quad
K = \pmat{0&0&I \\ -I& I&0},
}
then $\mcA_e = \mcA_o + B_{ei}K\mcC_e$ with $\Dom(\mcA_e)= \Dom(\mcA_o)$ and $\mcB_e = \mcB_o - Q_eK \mcC_e$. It is easy to check that $(\mcB_o,\mcA_o, \mcC_e, Q_e,B_{ei})$ is a well-posed boundary node. 
To apply \cref{prp:BCSfeedback}
 we analyse $I-KP_o(\gl)$, where $P_o$ is the transfer function of $(\mcB_o,\mcA_o, \mcC_e, Q_e,B_{ei})$. 
For $\gl\in\rho(\mcA_o\vert_{\ker(\mcB_o)})=\rho(A)$ a direct computation shows that 
\eq{
I-KP_o(\gl) 
= I- K\pmat{P(\gl)& 0\\ P(\gl)&P_L(\gl)\\ P_K(\gl) & P_{KL}(\gl)}
=  \pmat{I-P_K(\gl)&-P_{KL}(\gl) \\ 0&I- P_L(\gl)},
}
where $P_{KL}$ is the transfer function of $(\mcB ,\mcA,  \mcK,  L, L_i)$.
If $\gw> \max\set{\gw_0(T),\gb}$, then $P_{KL}\in H_\infty(\C_{\gw}^+; \Lin(Y,U))$ by  \cref{rem:WPbnodetransfun}  and our assumptions imply $(I-KP_o(\cdot))\inv \in H_\infty (\C_{\gw}^+; \Lin(U_e))$.
 \cref{prp:BCSfeedback} thus shows that $(\mcB_e,\mcA_e, \mcC_e, Q_e,B_{ei})$ is a well-posed boundary node on $\BNsp{U_{be}}{U_e}{X_e}{Y_e}$.

To study the stability of the semigroup $T_e$, define $S\in \Lin(X_e)$ such that 
$Sx = (x_1,x_2-x_1)\tp$ and $S\inv x = (x_1,x_1+x_2)\tp$ for $x=(x_1,x_2)\tp\in X_e$.
Then $\tilde A_e = S A_e S\inv$ with domain $\Dom(\tilde A_e)= S(\Dom(A_e))$ is the generator of the semigroup $ST_e S\inv$ and the stability properties of  $T_e$ and $ST_e S\inv$ coincide.
For $(x_1,x_2)\tp\in \Dom(\mcA)\times \Dom(\mcA)$ we have $x=(x_1,x_2)\tp\in \Dom(\tilde A_e)$ if and only if
\eq{
&S\inv x= (x_1 ,x_1+ x_2)\tp\in \Dom(A_e)= \ker(\BB_e)\\
 \Leftrightarrow \quad &
\begin{cases}
\mcB x_1 -Q\mcK (x_1 + x_2) =0\\
L\mcC x_1 +(\mcB-Q\mcK -L\mcC)(x_1 + x_2) =0
\end{cases}\\
 \Leftrightarrow \quad &
\begin{cases}
(\mcB-Q\mcK) x_1 -Q\mcK  x_2 =0\\
 (\mcB -L\mcC)x_2 =0.
\end{cases}
}
For $(x_1,x_2)\tp\in \Dom(\tilde A_e)$ a direct computation shows that
\eq{
\tilde A_e \pmat{x_1\\ x_2}
&= S \mcA_e\pmat{x_1\\ x_1 +x_2}
=  \pmat{(\mcA + B_i\mcK)x_1+ B_i\mcK x_2\\(\mcA + L_i \mcC)x_2}.
}
We will use \cref{prp:BNcascade} in the analysis of $\tilde A_e$.
Since $(I-P_K(\cdot))\inv \in H_\infty(\C_{\gb}^+;\Lin(U))$ and since $(\mcB,\mcA,\mcK,Q,B_i)$ is a well-posed boundary node,  \cref{prp:BCSfeedback} implies that 
$(\mcB-Q\mcK,\mcA+B_i \mcK,\mcK,Q,B_i)$ is a well-posed boundary node (and it in particular has a well-posed input map). 
Moreover, $(\mcB,\mcA,\pmatsmall{\mcC\\ \mcK},L,L_i)$ is a well-posed boundary node with  transfer function  $\pmatsmall{P_L\\ P_{KL}}$. We have $I-\pmat{I,\ 0}\pmatsmall{P_L \\ P_{KL}} = I-P_L$ and thus
 our assumption $(I-P_L(\cdot))\inv \in H_\infty(\C_{\gb}^+;\Lin(Y))$ and  
 \cref{prp:BCSfeedback} imply that $(\mcB-L\mcC,\mcA+L_i \mcC,\pmatsmall{\mcC\\ \mcK},L,L_i)$ is a well-posed boundary node (which has a well-posed output map). These properties in particular imply that $A_K$ and $A_L$ generate strongly continuous semigroups and the claims regarding the stability of $T_e$ follow from the structure of $\tilde A_e$ and \cref{prp:BNcascade}.
\end{proof}

Polynomial stability of both $T_K$ and $T_L$ leads to polynomial closed-loop stability for a class of  systems with collocated inputs and outputs.

\begin{proposition}
\label{prp:ObsStabIP}
Assume that $(\BB,\AA,\CC,I,0)$ is a well-posed boundary node on the Hilbert spaces $\BNsp{U_b}{U}{X}{U}$ and that $\re \iprod{\AA x}{x}_X\le \re \iprod{\BB x}{\CC x}_U $, $x\in \Dom(\AA)$. If we choose $\KK=-K\CC$, $K\in \Lin(U)$, $L\in \Lin(U)$, and $L_i=0$ so that $\re K\ge c_1I$ and $-\re L\ge c_2 I$ for some $c_1,c_2>0$, then 
$(\mcB_e,\mcA_e,\mcC_e,Q_e,B_{ei})$ is a well-posed boundary node on $\BNsp{U_{be}}{U_e}{X_e}{Y_e}$ and $A_K$ and $A_L$ generate strongly continuous semigroups $T_K$ and $T_L$, respectively. 
If $T_K$ and $T_L$ are polynomially stable with exponents $\polpar_K>0$ and $\polpar_L>0$, respectively, then $T_e$ generated by $A_e=\AA_e\vert_{\ker(\BB_e)}$ is polynomially stable with exponent $\polpar=\min \set{\polpar_K,\polpar_L}$.
\end{proposition}

\begin{proof}
Denote the transfer function of $(\BB,\AA,\CC,I,0)$ by $P_0$. 
The boundary node $(\mcB ,\mcA,  \pmatsmall{\mcC\\ \mcK}, [I,L], [0,0])$ is  well-posed, and 
the transfer functions $P_K$ and $P_L$ in \cref{thm:ObsStab} satisfy $P_K(\cdot)=-K P_0(\cdot)$ and $P_L(\cdot) = P_0(\cdot )L$. \cref{lem:IPBCSWP}(c) and \cref{thm:ObsStab} imply that
$(\mcB_e,\mcA_e,\mcC_e,Q_e,B_{ei})$ is a well-posed boundary node on $\BNsp{U_{be}}{U_e}{X_e}{Y_e}$ and $A_K$ and $A_L$ generate strongly continuous semigroups.
To analyse the stability of $T_e$ generated by $A_e$, 
we denote by $\Phi_t^K$, $t\ge0$, the input maps of 
$(\mcB-\mcK,\mcA,\mcC,I,0)$  and by
$\Psi_t^L$, $t\ge0$, the output maps of $(\mcB-L\mcC,\mcA, \mcK,L,0)$.
The proof of \cref{prp:BCSfeedback} and~\citel{HasPau25arxiv}{Thm.~2.5} (see also~\citel{Sta02}{Cor.~6.1}) imply that $\norm{\Phi_t^K}$ and $\norm{\Psi_t^L}$ are uniformly bounded with respect to $t\ge0$. 
If $\tilde A_e$ is as in the proof of \cref{thm:ObsStab}, then 
as in the proof of \cref{prp:BNcascade}
the semigroup $\tilde T_e$ generated by $\tilde A_e$ has the form $\tilde T_e(t)x=(T_K(t)x_1 + \Phi_t^K\Psi_t^Lx_2,T_L(t)x_2)\tp$, $(x_1,x_2)\tp \in X_e$, $t\ge 0$.
Since $T_K$ and $T_L$ are strongly stable,~\citel{Sta05book}{Thm.~8.1.6(a)} implies that for every $x\in X$ 
the function $u:\zinf\to U$ defined by $\Pt[t]u=\Psi_t^L x$, $t\ge0$, satisfies $u\in L^2(0,\infty;U)$
and $\Phi^K_t\Psi_t^L x\to 0$ as $t\to \infty$. Thus $\tilde T_e$ is strongly stable and bounded.
The proof of \cref{prp:BNcascade} shows $i\R\subset \rho(\tilde A_e)$ and
\eq{
(is-\tilde A_e)\inv = \pmat{(is-A_K)\inv & H_K(is)\KK (is-A_L)\inv \\0& (is-A_L)\inv}, \qquad s\in \R,
}
where $H_K$ is the transfer function of $(\mcB-\KK,\mcA, \mcC,I,0)$.
We have from~\citel{NicPau25}{Prop.~2.6} that $\norm{H_K(is)}\lesssim \norm{(is-A_K)\inv}^{1/2}$ and $\norm{\KK(is-A_L)\inv}\lesssim \norm{(is-A_L)\inv}^{1/2}$.
This and the structure of $(is-\tilde A_e)\inv$ imply that $\norm{(is-\tilde A_e)\inv}\lesssim 1+\abs{s}^{1/\polpar}$ for $\polpar=\min \set{\polpar_K,\polpar_L}$. Thus $\tilde T_e$ is polynomially stable with exponent $\polpar$ by~\citel{BorTom10}{Thm.~2.4}, and by similarity the same is true for $T_e$.
\end{proof}

As mentioned above, the controller~\eqref{eq:controller} does not need to have  well-defined dynamics when disconnected from the system~\eqref{eq:plant}. 
In particular,  the restriction of $\mcA+B_i\mcK + L_i\mcC$ to $\ker(\mcB-Q\mcK -L_i\mcC)$ is not required to generate a strongly continuous semigroup on $X$. Instead, we consider~\eqref{eq:controller} as a controller \emph{with an internal loop} in the sense 
of~\cite{WeiCur97,CurWei01,Mik02phd}. 
In this setting, the controller is considered to be equipped with two inputs $u_{c1}$, and $u_{c2}$ and two outputs $\hat y$ and $y_c$ so that 
\begin{subequations}
\label{eq:controllerIloop}
\eqn{
\dot{\hat x}(t) &= \mcA \hat x(t) + B_i u_{c1}(t) + L_i (\mcC \hat x(t) +u_{c2}(t)), \qquad \hat x(0) = \hat x_0\\
\mcB \hat x(t) &= Qu_{c1}(t) + L(\mcC \hat x(t) +u_{c2}(t)) \\
\hat y(t) &= \mcC \hat x(t)\\
y_c(t) &= \mcK \hat x(t) .
}
\end{subequations}
The system~\eqref{eq:plant} and the controller~\eqref{eq:controllerIloop} are connected  in two steps:
\begin{itemize}
\item[\textup{1.}] We connect $u_{c2}(t)=-y(t)+u_2(t)$, where 
 $u_2$ is an external input, and set $u_{c1}(t)=u(t)$
\item[\textup{2.}] We connect $u(t)= y_c(t) + u_1(t)$ 
in the system and the controller to arrive at the final closed-loop system in \cref{fig:contrscheme}.
\end{itemize}
The following result shows that under the assumptions of \cref{thm:ObsStab} the controller~\eqref{eq:controllerIloop} and  the system in Step 1 are associated with well-posed boundary nodes.

\begin{proposition}
\label{prp:ContrIntLoop}
Let the assumptions of Theorem~\textup{\ref{thm:ObsStab}} hold.
The controller~\eqref{eq:controllerIloop} is associated with a well-posed boundary node $(\mcB-L\CC ,\mcA + L_i\CC,  \pmatsmall{\mcC\\ \mcK}, [Q,L], [B_i,L_i])$ on $\BNsp{U_b}{U_e}{X}{Y\times U}$.
Moreover,
 the composite system consisting of~\eqref{eq:plant} and~\eqref{eq:controllerIloop}
with the connections
 $u_{c2}(t)=-y(t)+u_2(t)$ and $u_{c1}(t)=u(t)$ defines a well-posed boundary node on $\BNsp{U_{be}}{U_e}{X_e}{Y_e}$ with input
 $(u,u_2)\tp$ and output $(y,\hat y,y_c)\tp$.
\end{proposition}

\begin{proof}
Denote the transfer function of the boundary node $(\mcB ,\mcA,  \mcK,  L, L_i)$ by $P_{KL}$.
Since $(\mcB ,\mcA ,  \pmatsmall{\mcC\\ \mcK}, [Q,L], [B_i,L_i])$ is a well-posed boundary node by assumption and since $I-\pmatsmall{0\ 0\\[.7ex]I \ 0}\pmatsmall{P&P_L \\ P_K&P_{KL}} = \pmatsmall{I & 0\\[.7ex]-P& I-P_L}$,
 our assumption $(I-P_L(\cdot))\inv \in H_\infty(\C_{\gb}^+;\Lin(Y))$, \cref{rem:WPbnodetransfun}, and  
 \cref{prp:BCSfeedback} imply that $(\mcB-L\CC ,\mcA + L_i\CC,  \pmatsmall{\mcC\\ \mcK}, [Q,L], [B_i,L_i])$ is a well-posed boundary node.
If we define $x_0(t)=(x(t),\hat x(t))\tp $, $u_0(t)=(u(t),u_2(t))\tp$, and $y_0(t)=(y(t),\hat y(t),y_c(t))\tp$, then the 
system consisting of~\eqref{eq:plant} and~\eqref{eq:controllerIloop} on $X_e$ has the form
\eq{
\dot x_0(t) &= \mcA_1 x_0(t) + B_{ei} u_0(t)\\
\mcB_1 x_0(t) &= Q_e u_0(t) \\
y_0(t) &= \mcC_e x_0(t),
}
where $\AA_1:\Dom(\AA_1)\subset X_e\to X_e$ 
with domain $\Dom(\mcA_1)=\Dom(\mcA)\times \Dom(\mcA)$
and $\BB_1\in \Lin(\Dom(\AA_1),U_{be})$ are defined by
\eq{
\mcA_1 &= \pmat{ \mcA   & 0\\ -L_i \mcC& \mcA + L_i \mcC}, \quad
\mcB_1 = \pmat{\mcB &   0 \\L\mcC& \mcB- L\mcC}
}
and where $B_{ei}$, $Q_e$, and $\mcC_e$ are as in~\eqref{eq:ObsCLsysExtIO_operators2}. 
If we also let $\mcA_o$, $\mcB_o$ and $P_o$ be as in the proof of \cref{thm:ObsStab} and define $K=\pmatsmall{0&0&0\\ -I&I&0}$, then 
 $(\mcB_o,\mcA_o, \mcC_e, Q_e,B_{ei})$ is again a well-posed boundary node and 
 $\mcA_1 = \mcA_o + B_{ei}K\mcC_e$  and $\mcB_1 = \mcB_o - Q_eK \mcC_e$.
For $\gl\in\rho(A)=\rho(\mcA_o\vert_{\ker(\mcB_o)})$ we have
\eq{
I-KP_o(\gl) 
= I- K\pmat{P(\gl)& 0\\ P(\gl)&P_L(\gl)\\ P_K(\gl) & P_{KL}(\gl)}
=  \pmat{I&0 \\ 0&I- P_L(\gl)}.
}
Thus $(I-KP_o(\cdot))\inv \in H_\infty (\C_{\gw}^+; \Lin(U_e))$
and \cref{prp:BCSfeedback} shows that $(\mcB_1,\mcA_1, \mcC_e, Q_e,B_{ei})$ is a well-posed boundary node on $\BNsp{U_{be}}{U_e}{X_e}{Y_e}$.
\end{proof}

If we do not require external well-posedness of the closed-loop system, it is possible to achieve the semigroup property and stability of the closed-loop system under strictly weaker assumptions than in \cref{thm:ObsStab}. In this situation we consider the closed-loop system consisting of~\eqref{eq:plant} and~\eqref{eq:controller} with $u_1\equiv 0$ and $u_2\equiv 0$  and without any outputs.
The dynamics of the closed-loop state
 $x_e(t)=(x(t),\hat x(t))\tp\in X_e$ are then described by
\eq{
\dot x_e(t) &= \mcA_e x_e(t), \qquad x_e(0)= x_{e0}\\
\mcB x_e(t) &= 0,
}
where $A_e:= \mcA_e\vert_{\ker(\mcB_e)}$.
The following result presents conditions under which $A_e$ generates an exponential, strong, or polynomially stable semigroup. 
We denote the semigroups generated by $A_K$ and $A_L$ by $T_K$ and $T_L$, respectively.

\begin{corollary}
\label{cor:ObsStabSGonly}
Let $(\mcB ,\mcA ,  \mcC, Q, B_i)$ be a  boundary node on the Hilbert spaces $\BNsp{U_b}{U}{X}{Y}$.
Let $\mcK\in \Lin(\Dom(\mcA),U)$, $L\in \Lin(Y,U_b)$, and $L_i \in \Lin(Y,X)$ be such that 
$(\mcB -Q\mcK ,\mcA + B_i\mcK,  \mcC, Q, B_i)$ is a boundary node with a well-posed input map and 
 $(\mcB-L\mcC ,\mcA + L_i\mcC,  \mcK, Q, B_i)$ is a boundary node with a well-posed output map.
Then $A_e$ generates a strongly continuous semigroup $T_e$ on $X_e$ and the following hold. 
\begin{itemize}
\item[\textup{(a)}] If $T_K$ and $T_L$ are exponentially stable, then $T_e$ is exponentially stable with growth bound $\gw_0(T_e)=\max \set{\gw_0(T_K), \gw_0(T_L)}$.
\item[\textup{(b)}] If $T_K$ is strongly stable and $T_L$ is exponentially stable (or vice versa), then $T_e$ is strongly stable.
\item[\textup{(c)}] If $T_K$ is polynomially stable with exponent $\polpar>0$ and $T_L$ is exponentially stable (or vice versa), then $T_e$ is polynomially stable with exponent $\polpar>0$.
\end{itemize}
\end{corollary}

\begin{proof}
Let $S,S\inv\in \Lin(X_e)$ be as in the proof of \cref{thm:ObsStab}.
The operator $A_e$ generates a strongly continuous semigroup if and only if $\tilde A_e = S A_e S\inv$ with domain $\Dom(\tilde A_e)= S(\Dom(A_e))$ generates a semigroup and the stability properties of the two semigroups coincide.
The computations in the proof of \cref{thm:ObsStab} show that 
\eq{
\tilde A_e = \pmat{\mcA + B_i\mcK & B_i\mcK\\ 0 & \mcA + L_i\mcC}, \quad
\Dom(\tilde A_e) = \ker \left( \pmat{\mcB-Q\mcK & -Q\mcK \\ 0 & \mcB -L\mcC} \right).
}
The claims therefore follow from our assumptions and \cref{prp:BNcascade}.
\end{proof}

\begin{remark}
\label{rem:ObsContrNUstab}
Polynomial closed-loop stability generalises immediately to  non-uni\-form stability with decay rates described by general resolvent bounds on $i\R$. The proofs of \cref{thm:ObsStab} and \cref{cor:ObsStabSGonly} with \cref{rem:NUStab} imply that if $T_K$ is exponentially stable and $T_L$ is strongly stable so that $i\R\subset \rho(A_L)$ (or vice versa), the closed-loop resolvent satisfies 
$\norm{(is-A_e)\inv}\lesssim 1+\max\set{\norm{(is-A_K)\inv },\norm{(is-A_L)\inv }}$, $s\in \R$.
 This leads to decay rates of the classical closed-loop solutions via~\cite{BatDuy08,RozSei19}.
As the proof shows, this same estimate also holds under the assumptions of \cref{prp:ObsStabIP} in the case where
 $T_K$ and $T_L$ are strongly stable so that $i\R\subset \rho(A_K)$ and $i\R\subset \rho(A_L)$.
\end{remark}

\section{Stabilisation of PDE Systems}
\label{sec:PDEs}

\subsection{A Two-Dimensional Wave Equation}

In this example we consider a wave equation on the square $\Omega = [0,1]\times [0,1]$. The wave profile is controlled over the subdomain $\Omega_c=[1/4,3/4]\times [1/4,3/4]$, and the observation is taken over the rest of the domain, $\Omega_o=\Omega\setminus \Omega_c$.
The system is defined as
\begin{subequations}
\label{eq:2Dwave}
\eqn{
w_{tt}(\xi,t) &= \Delta w(\xi,t) + \chi_{\Omega_c}(\xi) u(\xi,t), & \xi\in \Omega, \ t>0\\
w(\xi,t) &= 0 ,  & \xi\in \partial \Omega , \ t>0\\
y(t) &= w_t(\cdot,t)\vert_{\Omega_o} & \phantom{\xi\in \partial \Omega,} \ t>0
}
\end{subequations}
with initial conditions $w(\cdot,0)= w_0\in H_0^1(\Omega)$ and $w_t(\cdot,0)= w_1\in L^2(\Omega)$, where   $\chi_{\Omega_c}$ is the characteristic function of $\Omega_c$.
The wave equation can be recast as a control system of the form~\eqref{eq:BCS} with state $x(t)=(w(\cdot,t),w_t(\cdot,t))\tp$ on $X=H_0^1(\Omega)\times L^2(\Omega)$ with norm $\norm{(x_1,x_2)\tp}_X^2=\norm{\nabla x_1}_{L^2}^2 + \norm{x_2}_{L^2}^2$, and with $U=L^2(\Omega_c)$ and $Y=L^2(\Omega_o)$. Since the wave equation does not have boundary inputs, we let $U_b=\set{0}$, $\mcB=0$, and $Q=0$. Moreover, we define $\mcA=A= \pmatsmall{0&I\\ \Delta & 0}$ with $\Dom(A)=(H_0^1(\Omega)\cap H^2(\Omega))\times H_0^1(\Omega)$ and define $B_i\in \Lin (U,X)$ and $\mcC\in \Lin(X,Y)$ so that
$B_iu = (0,\chi_{\Omega_c}u)\tp$, $u\in U$, and
 $\mcC x=x_2\vert_{\Omega_o}$ for $x=(x_1,x_2)\tp\in X$.
The  controller in the following result stabilises the wave equation. Since the control region $\Omega_c$ does not satisfy the Geometric Control Condition, exponential stabilisation 
is not possible.

\begin{proposition}
\label{prp:2Dwaveres}
The observer based controller
\eq{
\hat w_{tt}(\xi,t) &= \Delta \hat w(\xi,t) - \chi_{\Omega_c}(\xi) \hat w_t(\xi,t) 
- \chi_{\Omega_o}(\xi)(\hat w_t(\xi,t)- w_t(\xi,t)), && \xi\in \Omega\\
\hat w(\xi,t) &= 0 ,  && \xi\in \partial \Omega \\
u(t) &= -\hat w_t(\cdot,t)\vert_{\Omega_c}
}
with initial conditions $\hat w(\cdot,0)\in H_0^1(\Omega)$ and $\hat w_t(\cdot,0)\in L^2(\Omega)$
 stabilises the wave equation~\eqref{eq:2Dwave} in such a way that the closed-loop semigroup is polynomially stable with exponent $\polpar=1/2$.
\end{proposition}

\begin{proof}
The controller has the form~\eqref{eq:controller} with $u_1\equiv 0$ and $u_2\equiv 0$, with $L=0$, and with $\KK\in \Lin(X,U)$ and $L_i\in \Lin(Y,X)$ defined by 
$\KK x= - x_2\vert_{\Omega_c}$, $x=(x_1,x_2)\tp\in X$, and $L_iy = (0,-\chi_{\Omega_o}y)\tp$, $y\in Y$. In particular, $\KK=-B_i^\ast$ and $L_i=-\CC^\ast$. We have from~\citel{AnaLea14}{Thm.~2.3 \& p. 165} that
the semigroup $T_K$ is polynomially stable with exponent $\polpar=1/2$. Moreover, the pair $(\CC,A)$ is exactly observable (in some time $\tau>0$) by~\citel{TucWei09book}{Thm.~7.4.1}, and therefore $T_L$ generated by $A_L= A+L_i\CC=A-\CC^\ast \CC$ is exponentially stable by~\citel{ChiPau23}{Thm.~3.2 \& p.~1104}. Since $\mcB=0$, $\mcC\in \Lin(X,Y)$ and $\mcK\in \Lin(X,U)$,
$(\mcB -Q\mcK ,\mcA + B_i\mcK,  \mcC, Q, B_i)$ and 
 $(\mcB-L\mcC ,\mcA + L_i\mcC,  \mcK, Q, B_i)$ are well-posed boundary nodes
and the claim follows from \cref{cor:ObsStabSGonly}.
\end{proof}

\begin{remark}
\label{rem:2Dwavedisk}
The geometry of the spatial domain has a direct effect on the stabilisation of the wave equation~\eqref{eq:2Dwave}. For instance, we may let $\Omega$ be the unit disk $\Omega=B(0,1)$ and define the control and observation regions  as $\Omega_c=B(0,1/2)$ and $\Omega_o=\Omega\setminus \Omega_c$. 
Then the pair $(\mcC,A)$ is again exactly observable and $(A,B_i)$ is only approximately controllable since $\Omega_c$ does not satisfy the Geometric Control Condition. The choices 
 $L=0$ and $L_i=-\mcC^\ast$ lead to an exponentially stable $T_L$ by~\citel{TucWei09book}{Thm.~7.4.1}. With the choice $\mcK=-B_i^\ast \in \Lin(X,U)$ the semigroup $T_K$ is non-uniformly stable and its resolvent is exponentially bounded on $i\R$ by~\citel{Leb96}{Thm.~1}.
By \cref{prp:BNcascade} and the proof of \cref{thm:ObsStab} the closed-loop resolvent  has an exponential growth bound on the imaginary axis, and therefore the classical solutions of the closed-loop system decay with logarithmic rates as $t\to\infty$ by~\citel{BatDuy08}{Thm.~1.5}.
\end{remark}

\subsection{A One-Dimensional Wave Equation}

We consider the stabilisation of the undamped wave equation on the interval $(0,1)$, 
\begin{subequations}
\label{eq:Wave1D}
\eqn{
w_{tt}(\xi,t) &= w_{\xi\xi}(\xi,t), \qquad & \xi\in (0,1), \ t>0\\
-w_\xi(0,t) &=  u(t), \qquad  w_\xi(1,t) = 0&  t>0\\
y(t) &= \pmat{w(1,t)\\w_t(1,t)}&  t>0
}
\end{subequations}
with initial conditions $w(\cdot,0)=  w_0\in H^1(0,1)$ and $w_t(\cdot,0) = w_1\in L^2(0,1)$. The two Neumann boundary conditions lead to linearly growing solutions of the uncontrolled equation.
Moreover, the system cannot be stabilised with static output feedback because its input and output are located at the opposite ends of the spatial interval. Our main result below introduces an observer-based stabilising controller.

\begin{proposition}
\label{prp:Wave1D}
For every $\ell_b>0$ there exists $\ell_i^\ast>0$ such that if $\kappa_0,\kappa_1,>0$ and $0<\ell_i\le\ell_i^\ast$, then
the observer-based controller
\begin{subequations}
\label{eq:Wave1Dcontr}
\eqn{
\hat w_{tt}(\xi,t) &=\hat  w_{\xi\xi}(\xi,t) - \ell_i(\hat w(1,t)- w(1,t) + u_2(t))\\
-\hat w_\xi(0,t) &=  u(t), \qquad \hat  w_\xi(1,t) = -\ell_b (\hat w_t(1,t)- w_t(1,t)+u_3(t))\\
u(t) &= -\kappa_0\hat  w(0,t) - \kappa_1\hat  w_t(0,t) - \kappa_0\kappa_1 \int_0^1\hat w_t(\xi,t)d\xi + u_1(t)
}
\end{subequations}
with initial conditions $\hat w(\cdot,0)=  \hat w_0\in H^1(0,1)$ and $\hat w_t(\cdot,0) = \hat w_1\in L^2(0,1)$
stabilises the wave equation~\eqref{eq:Wave1D} exponentially.
The closed-loop system with input $(u_1(t),u_2(t),u_3(t))\tp$ and output $(w_t(0,t),w_t(1,t),\hat w_t(0,t))\tp$  is externally well-posed in the sense that for every $t>0$ there exists $M_t>0$ such that 
the generalised solutions of the closed-loop system satisfy
\eq{
\MoveEqLeft[8]\Norm{\pmat{w(\cdot,t)\\  \hat w(\cdot,t)}}_{H^1}^2 
+ \Norm{\pmat{ w_t(\cdot,t)\\  \hat w_t(\cdot,t)}}_{ L^2}^2 
+ \Norm{\pmat{w_t(0,\cdot)\\ w_t(1,\cdot)\\ \hat w_t(0,\cdot)}}_{\Lp[2](0,t)}^2 \\
&\le  M_t \left(  
\Norm{\pmat{w_0\\  \hat w_0}}_{H^1}^2 
+ \Norm{\pmat{ w_1\\  \hat w_1}}_{ L^2}^2 
+ \Norm{\pmat{u_1\\ u_2\\u_3}}_{\Lp[2](0,t)}^2 
\right).
}
\end{proposition}

\begin{proof}
The natural state space of the wave equation is $H^1(0,1)\times L^2(0,1)$. However, to deal with the generalised eigenvector associated to the spectral point $0$, we choose the states $x(t)$ and $\hat x(t)$ of the system and the controller as $x(t)=(\int_0^1 w(\xi,t)d\xi , w(\cdot,t)-\int_0^1 w(\xi,t)d\xi , w_t(\cdot,t))\tp$ and $\hat x(t)=(\int_0^1 \hat w(\xi,t)d\xi , \hat w(\cdot,t)-\int_0^1 \hat w(\xi,t)d\xi , \hat w_t(\cdot,t))\tp$ on the space $X=\C\times V \times L^2(0,1)$ with $V=\setm{f\in H^1(0,1)}{\int_0^1f(\xi)d\xi=0}$. 
We equip $X$ with the norm $\norm{(x_1,x_2,x_3)\tp}_X^2= \abs{x_1}^2 + \norm{x_2'}_{L^2(0,1)}^2+ \norm{x_3}_{L^2(0,1)}^2$.
We can recast the closed-loop system consisting of~\eqref{eq:Wave1D} and~\eqref{eq:Wave1Dcontr} with input
$u_e(t)=(u_1(t),u_2(t),u_3(t))\tp$ and output $y_e(t)=(w_t(0,t),w_t(1,t),\hat w_t(0,t))\tp$
as a boundary control system of the form~\eqref{eq:ObsCLsysExtIO}
on $X_e=X\times X$, $U=\C$, $Y=\C^2$, and $U_b=\C^2$
with operators of the form~\eqref{eq:ObsCLsysExtIO_operators} if we define
 $\AA : \Dom(\AA)\subset X\to X$, $\BB\in \Lin(\Dom(\AA),U_b)$, $\CC\in \Lin(\Dom(\AA),Y)$, $\KK\in \Lin(\Dom(\AA),U)$, $Q\in \Lin(U,U_b)$, $B_i\in \Lin(U,X)$, $L\in \Lin(Y,U_b)$, and $L_i\in \Lin(Y,X)$ by
\eq{
\AA x &= \pmat{\int_0^1 x_3(\xi)d\xi\\ x_3-\int_0^1 x_3(\xi)d\xi\\ x_2'' } , \quad
\BB x = \pmat{-x_2'(0)\\ x_2'(1)}, \quad
L_i = \pmat{0&0\\ 0& 0\\ -\ell_i \bm 1& 0}\\
\CC x &= \pmat{x_1+x_2(1)\\x_3(1)}, \quad
\KK x = -\kappa_0 (x_1 + x_2(0)) - \kappa_1 x_3(0) -\kappa_0 \kappa_1 \iprod{x_3}{\bm 1}_{\Lp[2]}
}
for $x=(x_1,x_2,x_3)\tp \in\Dom(\AA) := \C \times (H^2(0,1)\cap V)\times H^1(0,1)$
and
$B_i=0\in \Lin(U,X)$, $Q=\pmatsmall{1\\0}$, and $L = \pmatsmall{0& 0\\ 0&-\ell_b }$.
Employing an invertible change of coordinates $(x_1,x_2,x_3)\tp \to (x_2+x_1\cdot \bm{1},x_3)\tp$ and equivalence of norms, we have from~\citel{CurZwa20book}{Ex.~3.2.17} that $A=\AA\vert_{\ker(\BB)}$ generates a semigroup $T$ on $X$ with growth bound $\gw_0(T)=0$. 
The same equivalence of the norms shows that the claims of our result follow from \cref{thm:ObsStab} once we show that
$(\mcB ,\mcA,  \pmatsmall{\mcC\\ \mcK}, [Q,L], [B_i,L_i])$ is a well-posed boundary node on $\BNsp{U_b}{U_e}{X}{Y_e}$,
 that 
$(I-P_K(\cdot))\inv \in H_\infty(\C_{\gb}^+;\Lin(U))$ and $(I-P_L(\cdot))\inv \in H_\infty(\C_{\gb}^+;\Lin(Y))$ 
for some $\gb>0$,
and that the semigroups generated by $A_K:=\mcA\vert_{\ker(\mcB-Q\mcK)}$ and $A_L:=(\mcA+L_i\mcC)\vert_{\ker(\mcB-L\mcC)}$ are exponentially stable.

We begin by showing that $(\mcB ,\mcA,  \mcC_b, I, 0)$ is a well-posed boundary node on $\BNsp{U_b}{U_b}{X}{U_b}$, where $\CC_b\in \Lin(\Dom(\AA),U_b)$ is defined by $\CC_bx = (x_3(0),x_3(1))\tp$ for $x=(x_1,x_2,x_3)\tp\in \Dom(\AA)$.
We have $X=\C\times X_0$ with $X_0=V\times L^2(0,1)$ and with norm $\norm{(x_2,x_3)\tp}_{X_0}^2=  \norm{x_2'}_{L^2(0,1)}^2+ \norm{x_3}_{L^2(0,1)}^2$, and 
\eq{
\AA = \pmat{0&C_1\\0&\AA_0}, \qquad \BB = \pmat{0,\ \BB_0}, \quad \mbox{and} \quad \CC_b = \pmat{0, \ \CC_0},
}
with $\AA_0: \Dom(\AA_0)\subset X_0\to X_0$, $\BB_0,\CC_0\in \Lin(\Dom(\AA_0),U_b)$, and $C_1\in \Lin(X_0,\C)$ defined by
$ C_1x=  \int_0^1x_3(\xi)d\xi$,
\eq{
\AA_0 x &= \pmat{ x_3-\bm 1 \int_0^1 x_3(\xi)d\xi\\ x_2'' }, \quad
\BB_0 x = \pmat{-x_2'(0)\\ x_2'(1)}, \quad \mbox{and} \quad
\CC_0 x = \pmat{x_3(0)\\ x_3(1)}
}
for 
$x=(x_2,x_3)\tp \in\Dom(\AA_0) := (H^2(0,1)\cap V)\times H^1(0,1)$.
To show that $(\BB_0,\AA_0,\CC_0,I,0)$ is a well-posed boundary node on $\BNsp{\C^2}{\C^2}{X_0}{\C^2}$ we note that 
for $x=(x_2,x_3)\tp\in \Dom(\AA_0)$ integration by parts yields
\eq{
\re \iprod{\AA_0 x}{x}_{X_0}
&= \re \left[ \iprod{x_3'}{x_2'}_{L^2} + \iprod{x_2''}{x_3}_{L^2} \right]
= \re \left[ x_2'(1) \conj{x_3(1)} - x_2'(0) \conj{x_3(0)} \right]\\
&=\re \iprod{\BB_0 x}{\CC_0x}_{\C^2}.
}
Since $A=\AA\vert_{\ker(\BB)}$ generates a semigroup on $X$, the structures of $\AA$ and $\BB$ imply that $A_0:=\AA_0\vert_{\ker(\BB_0)}$ satisfies $\ran(\gl-A_0)=X_0$ for some $\gl>0$. 
Thus $A_0$ is maximally dissipative, and it generates a contraction semigroup on $X_0$ by the Lumer--Phillips theorem. Moreover, $\BB_0$ is clearly surjective, and since its codomain is finite-dimensional, $\BB_0$ has a bounded right inverse $\BB_0^r\in \Lin(U_b,\Dom(\AA_0))$. Thus $(\BB_0,\AA_0,\CC_0,I,0)$ is a boundary node on $\BNsp{U_b}{U_b}{X_0}{U_b}$.
We will show that its transfer function $P_0$ satisfies $\sup_{s\in\R}\norm{P_0(\gw+is)}<\infty$ for some $\gw>0$, which will imply well-posedness by 
\cref{lem:IPBCSWP}(a).
For $\gl\in\C_0^+$ and
 $u=(u_1,u_2)\tp\in \C^2$  we have $P_0(\gl)u=\CC_0x$, where
$x\in \Dom(\AA_0)$ is  such that $(\gl-\AA_0)x=0$ and $\BB_0x=u$.
Denoting
 $x=(x_2,x_3)\tp\in X_0$ and $u=(u_1,u_2)\tp\in \C^2$ and defining 
$z_2=x_2+\gl\inv \iprod{\bm{1}}{x_3}\bm{1}\in H^2(0,1)$ we have
\eq{
&\begin{cases}
(\gl-\AA_0)x = 0\\
\BB_0 x= u
\end{cases}
 \Leftrightarrow \quad
\begin{cases}
\gl z_2 - x_3  = 0\\
\gl x_3 - z_2''  = 0\\
-z_2'(0)= u_1, \quad z_2'(1)=u_2.
\end{cases}
}
Solving the last system of equations for $(z_2,x_3)\in H^2(0,1)\times H^1(0,1)$ and recalling that $P_0(\gl)u = (x_3(0),x_3(1))\tp$ shows that
\eq{
P_0(\gl) = \pmat{\frac{1}{\tanh(\gl)} & \frac{1}{\sinh(\gl)}\\
\frac{1}{\sinh(\gl)}&\frac{1}{\tanh(\gl)}}.
}
Since $\sup_{s\in\R}\norm{P_0(\gw+is)}<\infty$ for $\gw>0$,  $(\BB_0,\AA_0,\CC_0,I,0)$ is well-posed by \cref{lem:IPBCSWP}(a).
Moreover, since $C_1\in \Lin(X_0,\C)$,  \cref{prp:CascCoupWP}(a) implies that 
 $(\mcB ,\mcA,  \mcC_b, I, 0)$ is a well-posed boundary node on $\BNsp{U_b}{U_b}{X}{Y_e}$.
Finally, 
\eq{
\CC = \pmat{C_2 \\ [0,1]\CC_b} \quad \mbox{and} \quad \KK = [-\kappa_1,0] \CC_b + K_1,
}
where $C_2\in \Lin(X,\C)$ and $K_1 \in \Lin(X,U)$ are defined by $C_2x= x_1+ x_2(1)$ and
 $K_1 x = -\kappa_0 (x_1 + x_2(0)) - \kappa_0\kappa_1 \iprod{x_3}{\bm 1}_{L^2}$ for $x=(x_1,x_2,x_3)\tp\in X$.
These properties imply that 
$(\mcB ,\mcA,  \pmatsmall{\mcC\\ \mcK}, [Q,L], [B_i,L_i])$ is a well-posed boundary node on $\BNsp{U_b}{U_e}{X}{Y_e}$. 

We will next analyse the transfer functions  $P_K$ and $P_L$ in \cref{thm:ObsStab}. 
If we denote the transfer functions of $(\mcB ,\mcA,  \mcC_b, I, 0)$ by $H_b$ and $P_b$, then a direct computation shows that
 $P_b(\gl)=P_0(\gl)$,
$P_K(\gl) = [-\kappa_1,0]P_0(\gl) Q + K_1 H_b(\gl)Q $,
and 
\eq{
P_L(\gl) &=  \pmat{C_2H_b(\gl)\\ [0,1]P_0(\gl)}L + \CC (\gl-A)\inv L_i
}
 for all $\gl\in \C_0^+$.
Since $(\mcB ,\mcA,  \mcC_b, I, 0)$ is well-posed and $\gw_0(T)=0$, \cref{rem:WPbnodetransfun} implies that $\norm{H_b(\gl)}\le M_1 / \sqrt{\re \gl-1}$ and $\norm{\CC(\gl-A)\inv}\le M_2 / \sqrt{\re \gl-1}$ for some constants $M_1,M_2>0$ and for all $\gl\in \C_1^+$.
Because of this and the structures of $Q$ and $L$,
$(I-P_K(\cdot))\inv \in H_\infty(\C_\gb^+; \Lin(U))$ and
 $(I-P_L(\cdot))\inv \in H_\infty(\C_\gb^+; \Lin(Y))$ for some $\gb>1$ if the values of $\gl \mapsto 1+ \kappa_1/\tanh(\gl)$ and $\gl \mapsto 1+ \ell_b/\tanh(\gl)$ are uniformly bounded away from zero on some right half-plane.  
But this is true since $\re (1/\tanh(\gl))\ge 0 $ for all $\gl\in \C_0^+$.

Finally, we will analyse the semigroups $T_K$ and $T_L$ generated by $A_K=\mcA\vert_{\ker(\mcB-Q\mcK)}$ and $ A_L = (\AA+L_i\CC)\vert_{\ker(\BB-L\CC)}$, respectively. 
Exponential stability of $T_K$ follows from~\citel{KrsSmy08book}{Sec.~7.2}.
On the other hand, for a fixed $\ell_b>0$ we have
\eq{
A_L 
= \pmat{0&C_1   \\  L_{i0}& A_{L0}+ L_{i0}C_{d}},
}
where $A_{L0}=\AA_0\vert_{\ker(\BB_0-L\CC_0)}$, $L_{i0}=\pmatsmall{0\\-\ell_i}$, and  $C_{d}(x_2,x_3)\tp = x_2(1)$ for $(x_2,x_3)\tp \in X_0$.
Since $A_{L0}$ generates an exponentially stable semigroup by~\citel{CoxZua95}{Thm. 10.1} and since
$C_d\in \Lin(X_0,\C)$ and $\norm{L_{i0}}=\ell_i$, analysis of the block operator $A_L$ 
 can be used to show that there exists $\ell_i^\ast>0$ such that $T_L$ is exponentially stable for every $\ell_i\in \zabr{0}{\ell_i^\ast}$.
Crucially, explicit computations and Rouch\'e's theorem can be used to show that for all sufficiently small $\ell_i>0$ the 
Schur complement $\gl\mapsto S(\gl)=
\gl -C_1 (\gl-A_{L0}-L_{i0}C_d)\inv  L_{i0}C_{d1}
\in\C$ of $\gl-A_{L0}-L_{i0}C_d$ in the block operator $\gl-A_L$ only has zeros with negative real parts, and that its inverse is uniformly bounded on $\C_0^+$.
Together with \cref{thm:ObsStab} this completes the proof.
\end{proof}

\subsection{The SCOLE Model}
In this section we stabilise the SCOLE model~\cite{LitMar88b} which consist of an Euler--Bernoulli beam coupled to an ordinary differential equation modeling the dynamics of a tip mass.
Our equation 
\begin{align}\label{eq:scole}
\begin{cases}
\rho(\xi) w_{tt}(\xi,t)=-(EIw_{\xi\xi})_{\xi\xi}(\xi,t), \quad \xi\in (0,1), t>0\\
w(0,t)=0, \quad w_\xi(0,t)=0\\
mw_{tt}(1,t)-(EIw_{\xi\xi})_\xi(1,t)=0\\
Jw_{\xi tt}(1,t)+EI(1)w_{\xi\xi}(1,t)=u(t)\\
y(t)=EI(1)w_{\xi\xi}(1,t)
\end{cases}
\end{align}
with initial data $w(\cdot,0)\in H^2_\ell(0,1):= \setm{ f\in H^2(0,1)}{ f(0)=f'(0)=0 }$, $w_t(\cdot,0)\in L^2(0,1)$, $w_t(1,0)\in \C$, and $w_{\xi t}(1,0)\in \C$
can be used as a simplified model for the dynamics of a monopile wind turbine tower with a heavy nacelle~\cite{ZhaWei11c}. The torque input $u(t)$ acts on the tip mass and the output measures the bending moment at $\xi=1$.
The spatially varying parameters $\rho,EI\in C^4([0,1])$ are strictly positive and $m > 0$ and $J > 0$.

We will rewrite~\eqref{eq:scole} as a boundary control system of the form~\eqref{eq:plant}
with state $x(t)=(w(\cdot,t), w_t(\cdot,t), w_t(1,t), w_{\xi t}(1,t))\tp$
 on the
Hilbert space $X=H^2_\ell(0,1)\times L^2(0,1)\times \C ^2$. We equip $X$ with the norm
\eq{
    \|x\|_X^2=\int_0^1 EI(\xi)|w''(\xi)|^2d\xi+\int_0^1 \rho(\xi)|v(\xi)|^2d\xi+m|p|^2+J|q|^2}
for $x=(w,v,p,q)\tp\in X$.
Moreover, we define $U=\C$, $Y=\C$, and $U_b=\C$ and define  $\AA: \Dom(\AA)\subset X\to X$, $\BB\in \Lin(\Dom(\AA),U_b)$, and $\CC\in \Lin(\Dom(A),Y)$ by $\BB x=v'(1)-q$, $\CC x=EI(1)w''(1)$, and
\eq{
    \mathfrak A x=\begin{bmatrix}
        v\\ 
        -\rho^{-1}(EI w'')''\\ 
        m^{-1}(EI w'')'(1)\\ 
        -J^{-1}EI(1)w''(1)
    \end{bmatrix}
}
for
$x=(w,v,p,q)\tp\in \Dom(\AA)
    :=\setm{(w,v,p,q)\tp\in(H^2_\ell(0,1)\cap H^4(0,1))\times H^2_\ell(0,1)\times \mathbb{C}^2}{  p=v(1)}$.
Since the control acts through the finite-dimensional part of the model, we choose $Q=0\in \Lin(U,U_b)$ and define $B_i\in \Lin(U,X)$ by $B_iu=(0, 0, 0,  J^{-1}u )\tp$, $u\in U$.
With these choices~\eqref{eq:scole} has the form~\eqref{eq:plant}. Our first result establishes the well-posedness of the system.

\begin{proposition}\label{prp:scoleWP}
$(\mathfrak B,\mathfrak A,\mathfrak{C}, I, 0)$ is a well-posed boundary node on the Hilbert spaces  $\BNsp{U_b}{U}{X}{Y}$ and 
$\re\langle\mathfrak A x,x\rangle_X=\re\langle\mathfrak Bx,\mathfrak Cx\rangle_\C$
for all $x\in D(\mathfrak A)$.
\end{proposition}

\begin{proof}
By \cite[Eq. (22) and the subsequent paragraph]{Guo02c} the operator $A:=\mathfrak A\vert_{\ker(\BB)}$
is skew-adjoint and therefore generates a unitary group
on $X$.
 Moreover, we have $\mathfrak B,\mathfrak C\in\mathcal L(D(\mathfrak A),\C )$ and since
 $U_b$ is finite-dimensional, $\BB$ has a right inverse $\BB^r\in \Lin(U_b,\Dom(\AA))$. 
Therefore $(\mathfrak B,\mathfrak A,\mathfrak C,I,0)$ is a boundary node on $\BNsp{U_b}{U}{X}{Y}$.
For  $x_1=(w_1,v_1,p_1,q_1)\tp\in D(\mathfrak A)$ and
$x_2=(w_2,v_2,p_2,q_2)\tp\in D(\mathfrak A)$
 integrating by parts shows that
\begin{align}\label{eq:integrationbypart}
    \begin{aligned}
        \langle \mathfrak A x_1,x_2\rangle_X+ \langle x_1,\mathfrak A x_2\rangle_X&=\bigl(v_1'(1)-q_1\bigr)\overline{EI(1)w_2''(1)}  \\
            &\qquad+EI(1)w_1''(1)\overline{(v_2'(1)-q_2)}.
    \end{aligned}
\end{align}
Since $\mathfrak Bx_1=v_1'(1)-q_1$ and $\mathfrak Cx_1=EI(1)w_1''(1)$, setting
$x_2=x_1$ yields  the claimed identity
$\re\langle\mathfrak A x,x\rangle_X
=\re\langle\mathfrak Bx,\mathfrak Cx\rangle_\C$, $x\in \Dom(\AA)$.
Moreover, \cref{lem:IPBCSWP}(a) shows that 
 $(\mathfrak B,\mathfrak A,\mathfrak C,I,0)$ is well-posed provided that its transfer function $P$ satisfies
$\sup_{s\in\R}|P(\gw+is)|<\infty$ for some $\gw>0$.

For $\gl\in\C_0^+$ we have 
$P(\lambda)u=\mathfrak Cx$, where 
$x\in D(\mathfrak A)$ satisfies $(\lambda -\AA)x=0$ and
$\mathfrak Bx=u$. If we define $e_u=(0,0,0,u)\tp\in \Dom(\AA)$ and $z=x+e_u$, then $z\in \ker(\BB)=\Dom(A)$ and $\CC z=\CC x$.
By \cite[Lem.~3.1 and the subsequent paragraph]{Guo02c} the
eigenvectors of $A$ can be chosen to form an orthogonal basis
$(\phi_k)_{k\in\Z}$ of $X$. We write
$A\phi_k=i\mu_k\phi_k$, $i\mu_k\in i\R\setminus \set{0}$, and denote 
$\gamma_k:=\mathfrak C\phi_k$. 
Setting $x_1=x$ and $x_2=\phi_k\in D(A)$ in
\eqref{eq:integrationbypart} yields
\eq{
    \langle \mathfrak A x,\phi_k\rangle_X
    +
    \langle x,A\phi_k\rangle_X
    =
    \langle \mathfrak Bx,\mathfrak C\phi_k\rangle_\C = u \conj{\gg_k}.
}
Using $\AA x=\gl x$ and $A\phi_k=i\mu_k\phi_k$ implies $\langle x,\phi_k\rangle_X=\overline{\gamma_k}u/(\lambda-i\mu_k)$.
Similarly, the identity~\eqref{eq:integrationbypart} applied with $x_1=e_u\in \Dom(\AA)$ and $x_2=\phi_k\in D(A)$ together with $\AA e_u=0$ and $\BB e_u=-u$ shows that $\langle e_u,\phi_k\rangle_X=-i\overline{\gamma_k}u/\mu_k$. Since 
 the expansion of
$z$ in the basis $(\phi_k)_{k\in\Z}$ converges in the graph norm of $A$ and
since $\mathfrak C\vert_{\Dom(A)}\in\mathcal L(D(A),Y)$, we have
\eq{
    P(\lambda) u = \CC x = \CC z
= \sum_{k\in \Z} \iprod{z}{\phi_k} \gg_k
 =u\sum_{k\in \Z} \abs{\gg_k}^2 \left(  \frac{1}{\gl-i\mu_k}-\frac{i}{\mu_k}  \right).
}
By~\citel{Guo02c}{Lem.~2.3} there exist 
$N_0\in\N$ and $a,b,c>0$ such that
$a\abs{k}-b\le \sqrt{\abs{\mu_k}}\le a\abs{k}+c$
and $\mu_{-k}=-\mu_k$ whenever $\abs{k}\geq N_0$.
Moreover,~\citel{Guo02c}{Lem.~3.1} implies that $\sup_{k\in \Z}\abs{\gg_k}<\infty $.
Therefore~\citel{Reb95}{Thm.~2.7} implies $\sup_{s\in\R}\abs{P(\gw+is)}<\infty $ for every $\gw>0$, and thus
 $(\mathfrak B,\mathfrak A,\mathfrak{C}, I, 0)$
is well-posed by  \cref{lem:IPBCSWP}(a).
 \end{proof}

Our main result introduces a polynomially stabilising controller for~\eqref{eq:scole}.

\begin{proposition}
\label{prp:scoleMain}
   If  $\kappa>0$ and $\ell>0$, then the observer-based controller 
\eqn{
\label{eq:scoleObsContr}
\begin{cases}
\rho(\xi)\hat w_{tt}(\xi,t)=-(EI\hat w_{\xi\xi})_{\xi\xi}(\xi,t),\hspace{1.5cm} \xi\in(0,1), \ t>0,\\
\hat w(0,t)=0,\quad \hat w_\xi(0,t)=0\\
m\hat w_{tt}(1,t)-(EI\hat w_{\xi\xi})_\xi(1,t)= 0\\
J\dot{\hat q}(t)+EI(1)\hat w_{\xi\xi}(1,t)=-\kappa\hat q(t)+u_1(t)\\
J\hat w_{\xi t}(1,t)=J\hat q(t)-\ell EI(1)\bigl(\hat w_{\xi\xi}(1,t)- w_{\xi\xi}(1,t)\bigr)+\ell u_2(t)\\
u(t)=-\kappa\hat q(t)+u_1(t)
\end{cases}
}
with initial data $(\hat w(\cdot,0),\hat w_t(\cdot,0),\hat  w_t(1,0),\hat q(0))\tp \in H_\ell^2(0,1)\times L^2(0,1)\times \C^2 $
        stabilises the SCOLE system~\eqref{eq:scole} in such a way that the closed-loop semigroup is polynomially stable with exponent $\polpar=1/2$.
        Moreover, the closed-loop system with input
$(u_1(t), u_2(t))\tp$ and output $(w_{\xi\xi}(1,t), \hat w_{\xi\xi}(1,t))\tp$ is externally well-posed
    in the sense that for every $t>0$ there exists $M_t>0$ such that  the generalised solutions of the closed-loop system satisfy
    \eq{
    \begin{aligned}
        \MoveEqLeft[4] \|x(t)\|_X^2+\|\hat x(t)\|_X^2+
            \norm{(w_{\xi\xi}(1,\cdot),
            \hat w_{\xi\xi}(1,\cdot))\tp}_{L^2(0,t)}^2\\
        &\leq M_t\left(\|x(0)\|_X^2+\|\hat x(0)\|_X^2+
            \norm{(u_1,u_2)\tp}_{L^2(0,t)}^2
        \right),
    \end{aligned}
  }
where $\hat x(t)= (\hat w(\cdot,t), \hat w_t(\cdot,t),\hat w_t(1,t), \hat q(t))\tp$.
\end{proposition}

\begin{proof}
If we define $\KK\in \Lin(X,U)$ by $\KK x=-\kappa q$ for $x=(w,v,p,q)\tp\in X$ and choose
 $L=-\ell/J\in \Lin(Y,U_b)$ and $L_i=0\in \Lin(Y,X)$, then the abstract representation of~\eqref{eq:scoleObsContr} with state $\hat x(t)= (\hat w(\cdot,t), \hat w_t(\cdot,t),\hat w_t(1,t), \hat q(t))\tp\in X$ has the form of the observer-based controller~\eqref{eq:controller}. We will verify the assumptions in \cref{thm:ObsStab}.
We have from \cref{prp:scoleWP} that $(\mathfrak B,\mathfrak A,\mathfrak C,I,0)$ is well-posed, and since $\KK\in \Lin(X,U)$, $(\mcB ,\mcA,  \pmatsmall{\mcC\\ \mcK}, [Q,L], [B_i,L_i])$ is a well-posed boundary node on $\BNsp{U_b}{U_e}{X}{Y\times U}$.
We will next investigate the transfer functions $P_K$ and $P_L$ in \cref{thm:ObsStab}.
Since $Q=0$, we have 
$P_K(\lambda)=\mathfrak K(\lambda-A)^{-1}B_i$, $\gl\in \rho(A)$, and thus $\KK\in \Lin(X,U)$ and $\norm{(\gl-A)\inv}\le 1/\re\gl$ for $\gl\in\C_0^+$ imply  
    $(I-P_K(\cdot))^{-1}\in H_\infty(\C _\beta^+;\mathcal L( U ))$ for some $\beta>0$.
Moreover, if we  denote the transfer
    function of $(\mathfrak B,\mathfrak A,\mathfrak C,I,0)$ by $P_0$, then
    $I-P_L(\lambda)=I-P_0(\lambda)L=I+\frac\ell JP_0(\lambda)$ for $\gl\in\C_0^+$. Therefore \cref{prp:scoleWP} and \cref{lem:IPBCSWP}(c) imply that 
$(I-P_L(\cdot))\inv \in H_\infty(\C_{\gb}^+;\Lin(Y))$ for all $\gb>0$.

Finally, we will analyse the stability of the semigroups $T_K$ and $T_L$ generated by $A_K=(\AA+B_i\KK)\vert_{\ker(\BB)}$ and $A_L=\AA\vert_{\ker(\BB-L\CC)}$, respectively.
For $x=(w,v,p,q)\in D(\AA)$ the condition  $\mathfrak Bx=L\mathfrak Cx$ is equivalent to
    $Jq=Jv'(1)+\ell EI(1)w''(1)$. Because of this, the semigroup $T_L$ is exponentially stable by~\cite[Thm.~4.2]{Guo02c}.
On the other hand, the structure of $\KK$ together with~\cite[Thm.~3.3]{FkiPau26} implies that $T_K$ is polynomially stable with exponent $\polpar=1/2$.
The claims regarding the well-posedness and stability of the closed-loop system now follow from \cref{thm:ObsStab}.
\end{proof}

\subsection*{Declaration of AI use} Generative AI (Claude Sonnet 4.6) was used to assist in constructing the Rouch\'e argument in the proof of \cref{prp:Wave1D}. The authors verified all parts of the AI-assisted argument manually.

\end{document}